\documentclass{amsart}
\usepackage{amsmath,amssymb,bbm}
\usepackage{amsfonts,amssymb}

\usepackage[numeric]{amsrefs}

\usepackage{mathtools}
\usepackage{enumerate}
\usepackage{hyperref}

\usepackage[capitalise]{cleveref}

\usepackage{epsfig}

\def\Xint#1{\mathchoice
    {\XXint\displaystyle\textstyle{#1}}%
    {\XXint\textstyle\scriptstyle{#1}}%
    {\XXint\scriptstyle\scriptscriptstyle{#1}}%
    {\XXint\scriptscriptstyle\scriptscriptstyle{#1}}%
    \!\int}
\def\XXint#1#2#3{{\setbox0=\hbox{$#1{#2#3}{\int}$ }
        \vcenter{\hbox{$#2#3$ }}\kern-.6\wd0}}

\def\dashint{\Xint-}

\theoremstyle{plain} 

\newtheorem{thm}{Theorem}[section]
\newtheorem{cor}[thm]{Corollary}
\newtheorem{lem}[thm]{Lemma}
\newtheorem{prop}[thm]{Proposition}

\newtheorem{defn}[thm]{Definition}

\theoremstyle{remark}
\newtheorem{rem}[thm]{Remark}
\def\Ld{\Lambda}

\numberwithin{equation}{section}

\def\f{\frac}

\def\vi{\varphi}

\def\Bl{\Bigl}
\def\Br{\Bigr}

 \def\Ga{\Gamma}

\def\ta{\theta}
\def\al{{\alpha}}
\def\be{{\beta}}
\def\da{{\delta}}

  \def\ga{{\gamma}}
 \def\k{{\kappa}}

 \def\va{\varepsilon}

 \def\CC{{\mathbb C}}

 \def\NN{{\mathbb N}}

 \def\RR{{\mathbb R}}

 \def\osc{\operatorname{osc}}
  \def\sspan{\operatorname{span}}
  
  \def\dim{\operatorname{dim}}
   \def\det{\operatorname{det}}

  \def\sph{\mathbb{S}^{d-1}}
  
\def\Og{\Omega}

\def\al{\alpha}
\newcommand{\wt}{\widetilde}

\def\sub{\substack}

\def\p{\partial}

\def\bl{\bigl}
\def\br{\bigr}
\def\og{\omega}
\def\Ld{\Lambda}

\newcommand{\R}{{\mathbb{R}}}
\newcommand{\eps}{\varepsilon}

\newcommand{\diam}{{\rm diam}}
\newcommand{\dist}{{\rm dist}}

\begin{document}

\title[]{On Bernstein- and Marcinkiewicz-type inequalities on multivariate $C^\alpha$-domains}
\author{Feng Dai}
\address{Department of Mathematical and Statistical Sciences,
    University of Alberta, Edmonton, Alberta T6G 2G1, Canada}
\email{fdai@ualberta.ca}

\author{Andr\'as Kro\'o}
\address{Alfr\'ed R\'enyi Institute of Mathematics,
     and Department of Analysis, Budapest University of Technology and Economics,
    Budapest, Hungary}
\email{kroo@renyi.hu}

\author{Andriy Prymak}
\address{Department of Mathematics, University of Manitoba, Winnipeg, MB, R3T2N2, Canada}
\email{prymak@gmail.com}

\thanks{    The first author was supported by  NSERC of Canada Discovery
    grant RGPIN-2020-03909, the second author was supported by the NKFIH - OTKA Grant K128922, and the third author  was supported by NSERC of Canada Discovery grant RGPIN-2020-05357.
}


\keywords{Bernstein type inequality, Marcinkiewicz type inequalities, sampling  discretization, $C^\alpha$-domains, multivariate polynomials}
\subjclass[2010]{41A17, 41A10,  42C05, 46N10, 42B99}

\begin{abstract}
 We prove new Bernstein and Markov type inequalities in $L^p$ spaces associated with the normal and the tangential derivatives on the boundary of a general compact $C^\al$-domain with $1\leq \al\leq 2$. These estimates are also applied  to establish  Marcinkiewicz type inequalities  for discretization of $L^p$ norms of algebraic polynomials on $C^\al$-domains with asymptotically optimal  number of function samples used.

\end{abstract}

\maketitle

\section{Introduction and main results}

The classical Bernstein and Markov inequalities for univariate algebraic polynomials $f$ of degree $n$ give the following sharp upper bounds for their derivatives:
$$\|\sqrt{1-x^2}f'(x)\|_{C[-1,1]}\leq n\|f\|_{C[-1,1]}, \;\;\;\|f'\|_{C[-1,1]}\leq n^2\|f\|_{C[-1,1]}.$$
The above estimates were extended to the $d$ dimensional unit ball $B^d:=\{x\in \RR^d: |x|\leq 1\}$ by Sarantopoulos \cite{s} (Bernstein type estimate) and Kellogg \cite{ke} (Markov type estimate). Namely, it was shown therein that
\begin{equation}\label{c}
\|\sqrt{1-|x|^2}D f(x)\|_{C(B^d)}\leq n\|f\|_{C(B^d)}, \;\;\;\|Df\|_{C(B^d)}\leq n^2\|f\|_{C(B^d)},
\end{equation}
where $|x|=(\sum_{1\leq j\leq d}x_j^2)^{1/2}$ and $Df:=|\nabla f|=(\sum_{j=1}^d(\frac{\partial f}{\partial x_j})^2)^{1/2}$ is the $\ell_2$ norm of the gradient $\nabla f$ of $f\in \Pi^d_n$. Here $\Pi^d_n$ stands for the space of real algebraic polynomials of $d$ variables and total degree at most $n$. These inequalities have been instrumental
in proving various results in approximation theory, and, as a
result, they have been generalized and improved in many directions, in particular for various norms and multivariate domains. A crucial feature of the Bernstein-Markov estimates \eqref{c} lies  in the fact that the magnitude of derivatives is of order $n^2$ uniformly on the domain, and of order $n$ pointwise if a proper distance from the boundary of the domain is enforced. In the multivariate setting this phenomenon  becomes more intriguing in the sense that both the order $n^2$ and  the ``proper distance from the boundary" are affected by the
smoothness of the domain. In addition, the study of \emph{tangential} Bernstein-Markov inequalities is particularly important, since they play a significant role in the theory
of the so called \emph{optimal meshes} (see the corresponding definitions in Section 7 below), and the Marcinkiewicz-Zygmund type inequalities. In this respect let us mention that on
$C^\al$ domains $D$, $ 1\leq\al\leq 2$, the magnitude of tangential derivatives of  $p\in \Pi^d_n$ was shown to be of order $n^{\frac{2}{\al}}$ uniformly on the domain, see \cite{Kroo}.
In addition, the tangential derivatives at a given point $x\in D$ are of order $n$ if they are weighted by the quantity $d(x)^{\f 1 \al-\f12}$ with $d(x)$ being the distance from $x$ to the boundary of the domain (\cite{Kroo1}).

In a recent paper \cite{DaiPry} the authors proved a new tangential Bernstein type inequality in $L^p$ spaces for algebraic polynomials on a general
compact $C^2$ domain. As an important application of this inequality, a Marcinkiewicz type inequality with asymptotically optimal number of sample points
for discretization of $L^p$ norms of algebraic polynomials on $C^2$ domains was given. The main goal of the present paper is to extend the results of \cite{DaiPry} to $C^\al$ domains with $1\leq\al<2$ (see Definition \ref{def-C2} for the precise definition of $C^\al$-domains). Similar to the case of uniform norm for $C^\al$ domains, $ 1\leq\al<2$, both the order of the tangential derivatives in $L^p$ Markov type estimates and the measure of distance from the boundary of the domain in tangential $L^p$ Bernstein type upper bounds are affected by the quantity $\al$ characterizing the smoothness of the domain. 

In this introduction,  we shall describe  our main results and some basic  notations. Necessary details and
appropriate definitions will be   given  in later part of the paper.  Given a  nonempty set $E\subset \RR^d$,we
denote by  $\dist(\xi, E)$ the distance from a point $\xi\in\RR^d$ to a set  $E\subset \RR^d$; that is,  $\dist(\xi, E):=\inf_{\eta\in E}\|\xi-\eta\|$, (we define   $\dist(\xi, E)=1$ if   $E=\emptyset$), where $\|\cdot\|$ denotes the Euclidean norm. 
Let $\Pi^d_n$ denote  the space of all real  algebraic polynomials in $d$ variables  of total degree at most $n$. Given $\xi\in\RR^d$, we denote by  $\p_\xi f=(\xi\cdot\nabla) f$ the directional derivative of $f$ along the direction of $\xi$. Throughout this paper,  the letter  $c$  denotes a generic constant whose value may
change from line to line, and the notation  $A \sim B$ means that  there exists a constant $c>0$, called the constant of equivalence, such that  $c^{-1} B\leq A \leq c B$.

Let $\Og\subset \RR^d$ be  a compact $C^\al$-domain with $1\leq \al\leq 2$ and a nonempty boundary $\Ga=\p\Og$. 
Denote by $\mathbf{n}_\eta$  the  outer unit normal vector   to  $\Ga$ at  $\eta\in\Ga$.
For  $ \xi\in\Og$,  $ f\in C^\infty(\Og)$,
we define
$$ |\nabla_{\tan, \eta} f(\xi)|:= \max\Bl\{
\bl|  \p_{\pmb{\tau}} f(\xi)\br|:\   \   \pmb\tau\in\sph,\  \ \pmb\tau\cdot \mathbf n_\eta=0\Br\},\  \ \eta\in\Ga, $$
and 
\begin{equation}\label{eqn:D max def}
	\mathcal{D}_{n,\mu} f(\xi):=\max\Bl\{  |\nabla_{\tan, \eta} f(\xi)|:\   \ \eta\in\p \Og,\   \  \|\eta-\xi\|\leq \mu \vi_{n,\Ga}(\xi)^{2/\al}\Br\},\   \ \xi\in\Og,
\end{equation}
where $\mu\ge 1$ is a parameter, and 
\begin{equation*}
	\vi_{n,\Ga}(\xi):=\sqrt{\dist(\xi, \Ga)} +n^{-1},\   \ n=1,2,\dots, \xi\in\Og.
\end{equation*}

In this paper, we will prove the following generalization of the $L^p$ tangential Bernstein-Markov type inequality from $C^2$ to $C^\al$ domains:

\begin{thm}\label{thm-1-1A} 
	For    $0<p<\infty$,  any  $f\in\Pi_n^d$ and $\mu>1$,    we have
	$$  \Bl\| (\vi_{n,\Ga} )^{\f 2\al-1} \mathcal{D}_{n,\mu} f\Br\|_{L^p(\Og)} \leq C(\mu,\Og, p)  n  \|f\|_{L^p(\Og)},$$
	and 
	   \begin{equation}\label{1-3A} \Bl\|  \mathcal{D}_{n,\mu} f\Br\|_{L^p(\Og)} \leq C(\mu,\Og, p)  n^{\frac{2}{\al}}  \|f\|_{L^p(\Og)}.\end{equation}
\end{thm}

Such a generalization  requires first a weighted analogue of the $C^2$ result with an explicit dependence on the size of the second derivatives. This is accomplished in Section 2 for certain graph domains in $\RR^2$. In Section 3 we make a transition from the $C^2$ to $C^\al, 1\leq\al<2$ case using approximation by the Steklov transform (also known as Steklov function, see, e.g., \cite{Ac}*{Sections~67, 83}). Then in Section 4 we proceed by considering higher dimensional graph domains.
 Since general  $C^{\alpha}$ domains can be covered  by $C^{\alpha}$ graph  domains (referred to  $C^\al$-domains of special type in this paper), results proved in Section 4 allow us to establish Theorem \ref{thm-1-1A}, 
 the tangential  Bernstein type inequality, for   a  more general $C^\al$-domain with $1\leq \al\leq 2$. This is done in Section 5.

 As mentioned above the parameter $\al$ affects upper bounds in $L^p$ tangential Bernstein-Markov type inequality, so this naturally leads to the question of sharpness of Theorem \ref{thm-1-1A}. In Section 6, we provide an example of a $C^\al$ domain 
 $\Og$ for which the asymptotic upper bound in \eqref{1-3A} cannot be improved.

 Finally,   as an application, we show  
 in Sections 7 ($d=2$) and 8 ($d>2$) that, similar to \cite{DaiPry} our results yield  a Marcinkiewicz type inequality with an asymptotically optimal number of sample points  for discretizing  $L^p$ norms of algebraic polynomials on $C^\al$ domains in $\RR^d$ with $2-\frac{2}{d}<\al\leq 2$.  More precisely, we will prove 
the following Marcinkiewicz type inequalities:

\begin{thm} \label{thm-1-2A}If  $2-\frac{2}{d}<\alpha\leq 2$, then  for any positive integer $n$, there exists a partition $\Og=\cup_{1\leq j\leq N}\Og_j$  of $\Og$ with $N\leq c_\al n^d$ such  that for every $\xi_j\in \Og_j$, each  $f\in  \Pi^d_n$ and $d-1< p<\infty$,  we have
	$$\frac{1}{2}\sum_{j=1}^N |\Og_j||f(\xi_j)|^p\leq \int_{\Og}|f(\xi)|^pd\xi\leq 2\sum_{j=1}^N|\Og_j||f(\xi_j)|^p.$$
\end{thm}
 
 Of particular interest is the case of $d=2$, where Theorem \ref{thm-1-2A} provides Marcinkiewicz type inequalities for $L^p$ norms on a compact $C^\al$-domain for all $1<\al\leq 2$ and $1<p<\infty$.  Theorem~\ref{thm-1-2A} is proven in Section~7 for $d=2$ and in Section~8 for $d\ge 3$.

\section{Tangential Bernstein's inequality on $C^2$ domains of special type}

In this section we will verify a tangential Bernstein type inequality on certain $C^2$ graph domains which will be the starting point of our further considerations. The proof follows closely~\cite{DaiPry}*{Sect.~4.2}. We repeat all the steps here for completeness because we need to track the dependence of the constants involved on the geometry of the domain. In addition, we need to introduce a certain weight into the $L^p$ norm which will be crucial below. This weighted approach will require applications of the so called doubling weights. Recall that a non-negative measurable function $w:I\to\R$ on an interval $I\subset \RR$ is a doubling weight  with  constant $\beta$ if for any interval $J\subset I$
\[
\int_{2J\cap I}w(x)\,dx\le \beta \int_{J} w(x)\,dx,
\]
where $2J$ denotes the interval with the length twice that of $J$ and the same midpoint as $J$, see, e.g.~\cite{MT2}*{Sect.~2}. We will need the following one-dimensional $L^p$ Bernstein inequality with doubling weights which can be found in~\cite{MT2}*{Th.~7.3} or~\cite{Er}*{Th.~3.1}. If $w$ is a doubling weight on $[-1,1]$ with the constant $\beta$, then for any $0< p<\infty$, and any algebraic polynomial $f\in\Pi_n^1$
\begin{equation}\label{eqn:1d-weighted-bern}
    \int_{-1}^1 (1-x^2)^{p/2}|f'(x)|^pw(x)\,dx
    \le C n^p \int_{-1}^1|f(x)|^pw(x)\,dx,
\end{equation}
where $C$ depends only on $p$ and $\beta$.

Let $g\in C^2[-1,2]$, and $M:=\|g''\|_{C[-1,2]}+18$. Define
$$
G:=\{(x,y):\  \ 0\leq x\leq 1,\   \ 0\leq  g(x)-y\leq 1\},$$
and
$$
G_\ast:=\{(x,y):\  \ -1\leq x\leq 2, \   \  g(x)-4\leq y \leq g(x)\}.
$$
For $(x,y)\in G_*$, we set  $\tau_x:= (1, g'(x))\in \RR^2$, and $\da(x,y):=g(x)-y$. Furthermore, let us denote by $\p_u$ the operator $u\cdot \nabla$  of  directional  derivative  for each  $u\in \RR^d$.

\begin{thm}\label{thm-1-1:Bern}
For   any $f\in\Pi_n^2$,  $0<p<\infty$ and  $0<\va\leq 1$,
    \begin{align*}
&   \Bl( \iint_G | \p_{\tau_x} f(x,y)|^p (\va+\da(x,y))^{-\f12}\, dydx \Br)^{1/p}\\
    &\leq C_p\sqrt{M} n  \Bl( \iint_{G_*} |f(x, y)|^p ( \va+\da(x,y))^{-\f12}\, dy dx\Br)^{1/p},
    \end{align*}
    where      $C_p>0$ is a constant depending only on $p$.
\end{thm}

For the proof of Theorem \ref{thm-1-1:Bern}, we need to introduce some necessary notations. 
    Set  $r_0:= \sqrt{2/M}$, and  define
     $$E:=\Bl\{(z,t)\in\RR^2:\   \  z, z+t\in [-1,2],\   \   |t|\leq r_0\Br \}.$$
     We also   write
     $E=E^{+}\cup E^{-}$,
     where  $$E^{+} =\{(z,t)\in E:\  \ t\ge 0\}\   \   \  \text{and}\   \  E^{-} =\{(z,t)\in E:\ \ t\leq 0\}.$$
    For each fixed $z\in [-1,2]$,   define
     $$ Q_z(t):=g(z)+g'(z) t -\f A2t^2,\   \   \  t\in\RR,$$ where $A$ is a fixed constant between $5M/2$ and $3M$.
     By   Taylor's theorem,  we have that
     \begin{equation}\label{eqn:taylor}
     g(z+t)-Q_z(t) =\int_z^{z+t} [ A +g''(u)] (z+t-u)\, du, \  \  (z, t)\in E.
     \end{equation}

    Next,  we  define
 the differentiable mapping $\Phi: E\to \RR^2$ by
\begin{equation}\label{1-2:Phi}
\Phi(z, t)=(x,y): =\big(z+t, Q_z(t)\big),\   \     (z,t) \in E.
\end{equation}
 Denote by $\Phi_+$ and $\Phi_{-}$ the restrictions of $\Phi$ on the sets $E^+$ and $E^{-}$ respectively.

The following lemma summarizes some useful properties of the mapping $\Phi$, which will play a crucial role in the proof of Theorem ~\ref{thm-1-1:Bern}:

\begin{lem}\label{lem-1-1:Berns}
    \begin{enumerate}[\rm (i)]
        \item For each $(z,t)\in E$, we have that  $\Phi(z,t)\in G_\ast$ and
        \begin{equation}\label{1-3:Jacobian}\Bl| \det  \bl(J_\Phi(z,t)\br)\Br|\equiv \f {\p(x,y)}{\p(z,t)}= \bl(A+g''(z)\br) |t|,\end{equation}
        where $J_\Phi$ denotes the Jacobian matrix of the mapping $\Phi$.
        \item  Both the mappings $\Phi_+: E^{+}\to \RR^2$ and $\Phi_{-} : E^{-} \to \RR^2$ are  injective.
        \item  For each $(x,y)\in G$, there exists a unique $(z,t)\in E^{+}$ such that $0\leq t\leq r_1:=2/\sqrt{3M}$ and $\Phi(z,t)=(x,y)$.
    \end{enumerate}
\end{lem}
\begin{proof} (i) Let $(x,y) =\Phi(z,t) $ with $(z, t) \in E$.  Then $x=z+t\in [-1,2]$ and by \eqref{eqn:taylor},  we have
    \begin{equation}\label{1-4:Taylor} g(x)-y=g(z+t)-Q_z(t) =\int_z^{z+t} (g''(u)+A) (z+t-u)\, du,\end{equation}
    which implies that $g(x)\ge  y$.
    Since $|t|\leq \sqrt{2/M}$, we deduce  from \eqref{1-4:Taylor}  that
    $ g(x)-y\leq 2M t^2 \leq 4$, which  implies  that  $(x, y)\in G_\ast$.
    Finally,  we point out that   \eqref{1-3:Jacobian}  follows directly  from  \eqref{1-2:Phi} and straightforward calculations.

    (ii)  
     Assume that $\Phi(z_1, t_1) =\Phi(z_2, t_2)$ for some  $(z_1, t_1), (z_2, t_2) \in E$ with $t_1t_2\ge 0$ and $t_2\ge t_1$.  Then $z_1+t_1=z_2+t_2$ and $Q_{z_1}(t_1)=Q_{z_2} (t_2)$. We will show that $t_1=t_2$, which will in turn imply  $z_1=z_2$.
    Setting $\bar{x}=z_1+t_1$, we obtain from \eqref{1-4:Taylor}  that for $i=1,2$,
    \begin{align*}
    g(\bar{x}) -Q_{z_i}(t_i) =g(z_i+t_i)-Q_{z_i}(t_i) =\int_0^{t_i} [ g''(\bar{x}-v) +A] v \, dv,
    \end{align*}
    which implies that
    \begin{equation}\label{1-5:inj} \int_{t_1}^{t_2} \Bl( g''(\bar{x} -v)+A\Br) v\, dv =0.\end{equation}
    Since $g''(\bar{x} -v)+A\ge \f A2>0$ and $v$ doesn't change sign on the interval $[t_1, t_2]$,  \eqref{1-5:inj} implies that  $t_1=t_2$.   This proves that both $\Phi_+$ and $\Phi_{-}$ are  injective.

    (iii)    Given  $(x,y)\in G$, we define the function  $h: [0, r_1]\to \RR$ by
    $ h(t):= g(x) -Q_{x-t} (t)$.  Note that   the function $h$ is well defined on $[0, r_1]$   since $0\leq x\leq 1$ and   $-1<-r_1\leq x-t \leq 1$ for any $t\in [0, r_1]$.  Moreover, using \eqref{eqn:taylor}  with $z=x-t$, we have
    $$ h(t)=\int_{x-t} ^x (g''(u) +A) (x-u)\, du.$$
    Thus, $h$ is a continuous function on $[0, r_1]$ satisfying that $h(0)=0$ and
    $$ h(r_1) \ge \f 32 M  \int_0^{r_1} v\, dv =\f 34M r_1^2= 1.$$
    Since $0\leq g(x)-y\leq 1$, it follows  by the intermediate value theorem  that there exists $t_0\in [0, r_1]$ such that
    $h(t_0) =g(x)-y$. This implies that  $y=Q_{x-t_0} (t_0)$ and $(x,y) =\Phi(x-t_0, t_0)$.
        \end{proof}

\begin{proof}[Proof of Theorem ~\ref{thm-1-1:Bern}]
    Let $v_\va(x,y):= ( \va+\da(x,y))^{-\f12}$ for $(x,y)\in G_\ast$.
     Performing the change of variables $x=z+t$ and $y=Q_z(t)$, and using Lemma ~\ref{lem-1-1:Berns},  we obtain
        \begin{align*}
    I&:=    \Bl(\iint_{G} \Bl|\p_{\tau_x} f(x,y)|^p\, v_\va(x,y) dxdy\Br)^{1/p}\\
    & =\Bl(\iint_{\Phi_+^{-1} (G)} \Bl| \p_{\tau_{z+t}} f(z+t, Q_z(t))\Br|^p v_\va(z+t,Q_z(t))(A +g''(z)) t \, dzdt\Br)^{1/p}.
        \end{align*}
        A straightforward calculation shows that
        \begin{align*}
     \p_{\tau_{z+t}} f(z+t, Q_z (t)) = \f d{dt} \Bl[ f(z+t, Q_z(t))\Br] + w(z,t) \p_2 f(z+t, Q_z(t)),
        \end{align*}
        where
        $ w(z,t) =g'(z+t) -g'(z) +A t$.
    It then follows that  $I\leq C_p (I_1+I_2)$, where
    \begin{align}
        I_1&:=\Bl(\iint_{\Phi_+^{-1} (G)} \Bl| \f d {dt} \Bl[ f(z+t, Q_z(t))\Br]\Br|^pv_\va(z+t,Q_z(t)) (A +g''(z)) t\, dzdt\Br)^{1/p},\label{I1:def}\\
        I_2&:=  \Bl(\iint_{\Phi_+^{-1} (G)} |w(z,t)|^p |\p_2 f(z+t, Q_z(t))|^p v_\va(z+t,Q_z(t))(A +g''(z)) t \, dzdt\Br)^{1/p}.\notag
    \end{align}

    To  estimate the double integral $I_2$, we perform  the change of variables $x=z+t$ and $y=Q_z(t)$ once again to obtain that
    \begin{align}\label{1-7:Bern}
    I_2& =\Bl(\iint_G |u(x,y)|^p |\p_2 f(x,y)|^p v_\va(x,y)\, dxdy\Br)^{1/p},
    \end{align}
    where $u$ is a function on $\Phi(E^{+})$   given by
    $$ u\circ \Phi(z,t) =w(z,t),\    \    \   (z,t) \in E^{+}.$$
Note that according to Lemma \ref{lem-1-1:Berns} (iii), the function $u$ is well defined on $G$.
    We further  claim that
    \begin{equation}\label{1-6:Bern}
    |u(x,y)|\leq C \sqrt{M} \sqrt{g(x)-y},\   \    \  \forall (x,y)\in G.
    \end{equation}
    Indeed,  using  Lemma \ref{lem-1-1:Berns}  (ii), we may  write $x=z+t$ and $y=Q_z(t)$ with $(z,t) \in E^{+}$. We then have that
    $$|u(x,y)|=|w(z,t)|=|g'(z+t)-g'(z) +At |\leq C M |t|.$$
    On the other hand, using \eqref{eqn:taylor}, we have that
    \begin{equation}\label{1-9}
    g(x) -y =g(z+t) -Q_z(t) =\int_z^{z+t} ( g''(u) +A) (z+t-u) \, du \in \ge C M |t|^2.
    \end{equation}
    Combining these last two estimates, we deduce  the claim \eqref{1-6:Bern}.

    Now using \eqref{1-7:Bern} and \eqref{1-6:Bern},
    we obtain
    \begin{align*}
    I_2&\leq  C\sqrt{M}\Bl(\int_0^1 \Bl[ \int_{g(x)-1}^{g(x)} (\sqrt{g(x)-y})^p |\p_2 f(x,y)|^pv_\va(x,y)\, dy\Br]\, dx \Br)^{1/p}.
    \end{align*}
Note that for each fixed $x\in [0, 1]$, the function $$y\mapsto v_\va(x, y)=(g(x)-y+\va)^{-\f12}$$ is a doubling weight on the interval $[g(x)-4, g(x)]$ with a doubling constant independent of $\va\in (0, 1)$. Thus,  applying~\eqref{eqn:1d-weighted-bern} w.r.t. $y$ on the interval $[g(x)-4,g(x)]$, we have that
    \begin{align*}
    I_2&\leq C\sqrt{M} n\Bl(\int_0^1 \int_{g(x)-4}^{g(x)} |f(x,y)|^p v_\va(x,y)\, dydx\Br)^{1/p}.
    \end{align*}

    It remains to  estimate the integral $I_1$.
     Recall that $r_0=\sqrt{2/M}>r_1=\sqrt{4/(3M)}$.  Since $r_0\in (0, \f13]$, it's easily seen from  Lemma ~\ref{lem-1-1:Berns} (iii)  that
    \begin{equation}
    \Phi_{+}^{-1} (G) \subset [-r_1,1]\times [0, r_1]\subset [-r_0, 1] \times [-r_0, r_0] \subset E.
    \end{equation}
    By the equality from~\eqref{1-9} and the choice of $A$, we also have
    \begin{equation*}
     (\va+2M t^2)^{-\f12}\leq v_\va(x, y) \leq  \Bigl(\va+\f 34M t^2\Bigr)^{-\f12},
    \end{equation*}
and both sides are equivalent up to a constant factor to $(\va+ M t^2)^{-\f12}$.
        Thus, using \eqref{I1:def},     we  obtain
        \begin{align*}
        I_1     &\leq C M^{1/p} \Bl( \int_{-r_1 }^{1} \Bl[ \int_{0}^{r_1}\Bl| \f d{dt} \Bl[ f(z+t, Q_z(t))\Br] \Br|^p  (\va+ M t^2)^{-\f12} |t|\, dt\Br]\, dz\Br)^{1/p}.
    \end{align*}
A straightforward calculation shows that for any $\va'>0$,   $|t|(\va'+ t^2)^{-\f12}$ is a doubling weight on $[-1,1]$ with a  doubling constant independent of $\va'$,
which also implies that the doubling constant of $|t|(\va+ M t^2)^{-\f12}$ is independent of both $M$ and $\va>0$.
 Now using~\eqref{eqn:1d-weighted-bern} w.r.t. $t$ on the interval $[-r_0,r_0]$, we obtain
    \begin{align*}  I_1 &\leq CM^{1/p} n r_0^{-\f12} (r_0-r_1)^{-\f12}\Bl( \int_{-r_0 }^{1} \Bl[ \int_{-r_0}^{r_0}\Bl| f(z+t, Q_z(t)) \Br|^p |t|(\va+ M t^2)^{-\f12}\, dt\Br]\, dz\Br)^{1/p}\\
    &   \leq CM^{\f 12 +\f1p}  n \Bl(\iint_{E} |f(z+t, Q_z(t))|^p|t|(\va+ M t^2)^{-\f12}\, dzdt\Br)^{1/p}. \end{align*}
        Splitting this last double integral into two parts $\iint_{E^{+}}+\iint_{E^{-}}$, and applying  the change of variables $x=z+t$ and $y=Q_z(t)$ to each of them  separately, we obtain that
    \begin{align*}
I_1&    \leq   C_pM^{1/2}n \Bl(\iint_{E^{+}\cup E^{-}} |f(z+t, Q_z(t))|^p(A +g''(z)) |t| v_\va(z+t, Q_z(t))\, dzdt\Br)^{1/p} \\
        &\leq C_p M^{1/2} n \Bl[ \iint_{\Phi(E^{+})}  |f(x,y)|^p v_\va(x,y) dxdy+\iint_{\Phi(E^{-})}  |f(x,y)|^p v_\va(x,y) dxdy \Br]^{1/p}\\
        &\leq C_p M^{1/2} n  \Bl( \int_{-1}^2 \int_{g(x)-4}^{g(x)} |f(x, y)|^p \f 1 { \sqrt{\va+\da(x,y)}}\, dy dx\Br)^{1/p},
        \end{align*}
        where the last step uses Lemma  \ref{lem-1-1:Berns} (i).

        \end{proof}

    \section{Tangential Bernstein's inequality on $C^\al$ domains of special type}

          The next step of our construction is to proceed from the $C^2$ to $C^\al$ graph domains in $\RR^2$. This will be done using approximation of $C^\al$ functions by their Steklov transform.  Let   $g: [-4,4]\to \RR$  be  a $C^\al$-function  with constant $L>1$ for some $1\leq \al\leq 2$; that is,  $g\in C^1[-4, 4]$ and      $$|g'(x+t)-g'(x)|\leq L |t|^{\al-1}\   \   \text{ whenever $x, x+t \in [-4,4]$}.$$
    Define  \begin{align*}
    G:&=\{(x,y):\  \ 0\leq x\leq 1,\   \ g(x)-1\leq y \leq g(x)\},\\
    G_\ast:&=\{(x,y):\  \ -1\leq x\leq 2, \   \  g(x)-8\leq y \leq g(x)\},\\
    \tau_x:&=(1, g'(x))\   \   \ \text{and}\   \  \da_n(x,y)=\da(x,y)+\f 1{n^2}=g(x)-y +\f 1 {n^2},\   \   \   \   (x,y)\in G.
    \end{align*}
    Here we point out that the set $G_\ast$ defined here is different from the set $G^\ast$ introduced in the last section. 
    We will keep these assumptions and notations throughout this section. We will show that \cref{thm-1-1:Bern} implies the following:

    \begin{thm}\label{thm-2-2}
            Let  $\ga=\f 1 \al-\f12$.
        Then for any  $f\in\Pi_n^2$ and $0<p<\infty$,
        $$ \Bl(\iint_G  (\da_n(x,y))^{\ga p} | \p_{\tau_x} f(x,y)|^p\, dydx\Br)^{1/p} \leq C_p L^{\f 1 \al}  n \|f\|_{L^p(G_\ast)}.$$

    \end{thm}

    The proof of  this theorem relies on  the following lemma:

    \begin{lem}\label{lem-1-2:Bern} For any  $0<b\leq 1$,  $0<p<\infty$, and   $f\in\Pi_n^2$,   we have
        \begin{align*}
    &   \Bl(\int_0^1 \Bl[\int_{g(x)-1} ^{g(x)} | \p_{\tau_x} f(x,y)|^p\, \f {dy}{\sqrt{\da(x,y)+b}}\Br]dx\Br)^{1/p}\\
    & \leq C(p,\al) L^{\f 1 \al}  b^{\f12-\f 1\al}n\Bl( \int_{-1}^2 \Bl[\int_{g(x)-6}^{g(x)+\f b4}|f(x,y)|^p\, \f {dy}{\sqrt{\da(x,y)+b}}\Br]\, dx\Br)^{1/p}.
        \end{align*}
    \end{lem}

    \begin{proof}
        Let $0<\da<1$ be a small parameter to be specified shortly.  For the function $g$ its Steklov transform is given by
        $$
        g_{\delta}(x):=\frac{1}{4 \delta^{2}} \iint_{[-\delta, \delta]^2} g(x+u+v) d u d v, \quad   \ -1\leq x\leq 2.
        $$
        Then   $g_\da\in C^2[-1,2]$, and moreover,   for any $x\in [-1, 2]$, we have
        \begin{align*}
        |g(x)-g_\da(x)| &=\f 1 {8\da^2} \Bl|\iint_{[-\da, \da]^2} \int_0^{u+v} \Bl( g'(x+s)-g'(x-s)\Br) \, ds dudv \Br|\leq \wt C L \da^\al, \\
        |g'(x)-g_\da'(x)|&\leq \f 1 {4\da^2} \iint_{[-\da, \da]^2} |g'(x+u+v)-g'(x)|\, dudv \leq \wt CL \da^{\al-1}, \\
        |g_\da''(x)|&\leq \f 1{4\da^2} \int_{-\da}^\da |g'(x+u+\da)-g'(x+u-\da)|\, du \leq \wt C L \da^{\al-2}.
        \end{align*}
        Now we  choose   $\da>0$ so that  $\wt CL\da^\al=\f b {8}$.

        Next, for $x\in [0,1]$, we set $ \tau_x^\da : =(1, g_\da'(x))$.
        Then for $f\in\Pi_n^2$ and $(x, y)\in G$,   \begin{align*}
        |\p_{\tau_x} f(x, y)| &=|\p_1f(x,y) +g'(x) \p_2 f(x, y)| \leq |\p_{\tau_x^\da} f(x,y)| + \wt C L\da^{\al-1} |\p_2 f(x, y)|\\
        &= |\p_{\tau_x^\da} f(x,y)| + C L^{1/\al} b^{1-\f 1\al}|\p_2 f(x, y)|.
        \end{align*}
        Thus,
        \begin{align*}
        \Bigl(\iint_G |\p_{\tau_x} f(x,y)|^p \, \f {dxdy}{\sqrt{\da(x,y)+b}}\Bigr)^{1/p} & \leq C\Bl( J_1+ L^{1/\al} J_2\Br),\end{align*}
        where
        \begin{align*}
            J_1:&= \Bl( \int_0^1 \Bl[\int_{g(x)-1}^{g(x)} |\p_{\tau_x^\da } f(x,y)|^p \f {dy}{\sqrt{\da(x,y)+b}}\Br]dx \Br)^{1/p},\\
        J_2:&= b^{1-\f 1\al} \Bl( \int_0^1 \Bl[\int_{g(x)-1}^{g(x)} |\p_2 f(x,y)|^p\, \f {dy}{\sqrt{\da(x,y)+b}}\Br] dx \Br)^{1/p}.
        \end{align*}

        On the one hand, for each fixed $x\in [0,1]$ and each $0<b\leq 1$, the function
        $$y\mapsto ( g(x)-y+b)^{-\f12}=(\da(x,y)+b)^{-\f12}$$
         is a doubling weight on $[g(x)-2, g(x)]$ with the doubling constant independent of $b$. Thus,
         using~\eqref{eqn:1d-weighted-bern}, we obtain
        \begin{align*}
    J_2\leq  C n b^{\f 12-\f1\al} \Bl( \int_0^1\Bl[ \int_{g(x)-2}^{g(x)+ \f b8
    } | f(x,y)|^p\, \f {dy}{\sqrt{\da(x,y)+b}}\Br]dx \Br)^{1/p}.
        \end{align*}
        On the other hand, since $\|g-g_\da\|_\infty \leq \wt C L\da^\al=\f b {8}$,
        we have
        \begin{align*}
        J_1
        &\leq  C\Bl( \int_0^1 \Bl[\int_{g_\da(x)-1-\f b{8}}^{g_\da(x)+\f b{8}} |\p_{\tau_x^\da } f(x,y)|^p \f {dy}{\sqrt{g_\da(x)-y+b}}\Br]dx \Br)^{1/p}.
        \end{align*}
        Using Theorem  \ref{thm-1-1:Bern} with $M=C L\da^{\al-2}$, we obtain
        \begin{align*}
    J_1 &\leq C \sqrt{L} n  \da^{\f \al 2-1} \Bl( \int_0^1\Bl[ \int_{g_\da(x)-5-\f b{8}}^{g_\da(x)+\f b{8}} |f(x,y)|^p \f {dy}{\sqrt{g_\da(x)-y+b}}\Br]dx \Br)^{1/p}\\
        &\leq C L^{1/\al}  b^{\f12-\f 1\al}n\Bl( \int_0^1 \Bl[\int_{g(x)-6}^{g(x)+\f b4} |f(x,y)|^p \f {dy}{\sqrt{g(x)-y+b}}\Br]dx \Br)^{1/p}.
        \end{align*}
        Putting the above estimates together, we get the desired inequality.        \end{proof}

    \begin{proof}[Proof of Theorem \ref{thm-2-2}]
        Since
        $$\p_{\tau_x} f(x,y)=\p_1 f(x,y)+g'(x) \p_2 f(x,y)$$
        is a polynomial of degree at most $n$ in the variable $y$, we obtain by the Remez inequality (see~\cite{MT2}*{(7.17)} and~\cite{Er}*{Sect.~7}) that
        \begin{align*}
        & \iint_G ( \da_n(x,y))^{\ga p} | \p_{\tau_x} f(x,y)|^p\, dxdy\leq C\int_0^1 \int_{g(x)-1}^{g(x)- \f {2} {n^2}} (\da(x,y))^{\ga p} | \p_{\tau_x} f(x,y)|^p\, dydx\\
        &\leq C \sum_{1\leq j<2\log_2 n} \Bl( \f {2^j}{n^2}\Br)^{\ga p}  \int_0^1 \Bl[\int_{g(x)-\f {2^{j+1}} {n^2}}^{g(x)- \f {2^j} {n^2} } | \p_{\tau_x} f(x,y)|^p  \, dy\Br] dx\\
            &\leq C \sum_{1\leq j<2\log_2 n} \Bl( \f {2^j}{n^2}\Br)^{\ga p+1/2}  \int_0^1 \Bl[\int_{g(x)-1-\f{2^j}{n^2}}^{g(x)- \f {2^{j}} {n^2} } | \p_{\tau_x} f(x,y)|^p  \Bl(\da(x,y)\Br)^{-1/2}\, dy\Br] dx.\end{align*}
    For each integer  $1\leq j<2\log_2 n$,  we apply Lemma \ref{lem-1-2:Bern} to  $b=\f {2^{j+1}} {n^2}$ and the function $g(x)-\f {2^j}{n^2}$ instead of $g$. We then obtain
        \begin{align*}
     \iint_G ( \da_n&(x,y))^{\ga p} | \p_{\tau_x} f(x,y)|^p\, dxdy \\&  \leq C L^{\f p\al} n^p \sum_{1\leq j<2\log_2 n}  \Bl( \f {n^2}{2^j}\Br)^{ \ga p } \Bl( \f {2^j}{n^2}\Br)^{\ga p+1/2}
         \int_{-1}^2 \Bl[\int_{g(x)-8 }^{g(x)-\f {2^{j-1}}{n^2} } | f(x,y)|^p\Bl(\da(x,y)\Br)^{-1/2}\, dy\Br] dx\\
        &\leq C L^{\f p\al} n^p \int_{-1}^2 \Bl[\int_{g(x)-8 }^{g(x)-\f 1{n^2} } | f(x,y)|^p \sum_{1\leq j \leq \log_2(2n^2 \da(x,y)) } \Bl( \f {2^j} {n^2 \da(x,y)}\Br)^{1/2}dy \Br] dx\\
        &\leq C L^{\f p\al} n^p \int_{-1}^2 \int_{g(x)-8 }^{g(x) } | f(x,y)|^pdydx.
        \end{align*}

    \end{proof}

    \section{Tangential Bernstein's inequality: Higher-dimensional case}

    In order to establish the corresponding  result on a general $C^\al$-domain, we need to formulate a slightly stronger version of the  tangential Bernstein inequality on high dimensional domains of special type. Let us first introduce the general set up for the rest of this section.

    We will     write a point in $\RR^d$ in the form  $(x, y)$ with $x\in\RR^{d-1}$ and $y=x_d\in \RR$. Sometimes for the sake of simplicity,   we  also  use a Greek letter to denote a point in $\RR^d$ and write it in the form $\xi=(\xi_x, \xi_y)$ with $\xi_x\in\RR^{d-1}$ and $\xi_y\in\RR$.

Let $D_1:=[0,1]^{d-1}$ and $D_2=[-1,2]^{d-1}$, where $d\ge 3$.  Let   $1\leq \al \leq 2$, and let   $g$  be a $C^\al$-function on $[-4, 4]^{d-1}$   with constant $L>1$:   $$\|\nabla g(x+t)-\nabla g(x)\|\leq L \|t\|^{\al-1}\   \   \text{ whenever $x, x+t \in [-4,4]^{d-1}$},$$  where $\|\cdot\|$ denotes the Euclidean norm.
 Define
$$G:=\Bl\{(x,y):\  \  x\in D_1,\   \    g(x)-1\leq y\leq g(x)\Br\},$$
and
$$ G_\ast:=\Bl\{(x,y):\  \ x\in D_2, \  g(x)-8\leq y\leq g(x)\Br\}.$$
We call the set
\begin{align*}
\p' G :&=\{ (x, g(x)):\  \   x\in D_1\},
\end{align*}
the essential boundary of $G$. We also define $\p' G_\ast :=\{ (x, g(x)):\  \   x\in D_2\}$.

Given a point $x_0\in D_1$, we denote by $S_{x_0}$ the space of all tangent vectors to the surface $y=g(x)$, $x\in D_2$ at the point $(x_0, g(x_0))$; namely,
$$ S_{x_0}:=\Bl\{ \xi\in\RR^{d}:\   \   \xi\cdot \big(\nabla g(x_0),-1\big)=0 \Br\}.$$
 Let $e_1=(1,0,\cdots, 0)$, $\cdots$, $e_{d} =(0,\cdots, 0, 1)$ denote   the standard  canonical basis of  $\RR^{d}$.
For $j=1,2,\cdots, d-1$,   we define  \begin{equation}\label{11-4-0}
\xi_j (x_0) :=e_j + \p_j g(x_0) e_{d}.
\end{equation}
Geometrically,  $\xi_j(x_0)$ is a tangent vector to the essential boundary  $\p' G$ at the point  $(x_0, g(x_0))$ that is parallel to   the $x_jx_d$-coordinate plane.
It is easily seen that
\begin{equation}\label{2-2}
S_{x_0}=\Big\{ \big (\eta,\   \eta\cdot \nabla g(x_0)\big):\  \  \eta\in\RR^{d-1}\Br\} =\sspan\Bl \{ \xi_j(x_0):\   \  1\leq j\leq d-1\Br\}.
\end{equation}

Next, for $\xi=(\xi_x, \xi_y)\in G$, we define
\[ \da_n(\xi): ={g(\xi_x)-\xi_y}+\f1{n^2},\   \   n=1,2,\cdots,\]
and set, for  $\mu>1$,
$$\Xi_{n, \mu,\al} (\xi):= \Bl\{ u\in D_1:\  \  \|u-\xi_x\|\leq \mu (\da_n(\xi))^{1/\al}\Br\}.$$
 For $f\in C^1(G_\ast)$ and $u\in D_1$, define
\[ |\nabla_{\tan, u} f(\xi)|:=\max_{\eta\in S_u\cap \sph} |\p_\eta f(\xi)|,\   \ \xi\in G,\]
where $\sph$ denotes the unit sphere of $\RR^d$.
Using \eqref{2-2}, we have \begin{equation}
|\nabla_{\tan, u} f(\xi)|\sim \sum_{j=1}^{d-1} |\p_{\xi_j(u)} f(\xi)|,  \   \  \xi\in G, \   \ u\in D_1.
\end{equation}
Note that if $\xi=(\xi_x, \xi_y) \in \p'G$ and $u=\xi_x$, then $ |\nabla_{\tan, \xi_x} f(\xi)|$ is the length of the tangential gradient of $f$ at $\xi$. In the following result, we choose $u$ to be an arbitrarily given point in a certain neighborhood of $\xi_x$.

\begin{thm}\label{thm-3-2}
    Let  $\ga=\f 1 \al-\f12$, and   $0<p<\infty$.
    Then for any  $f\in\Pi_n^2$ and parameter $\mu>1$,    we have
    \begin{equation}\label{3-4}
    \Bl\| (\da_n(\xi)) ^\ga \max_{ u\in \Xi_{n,\mu,\al}(\xi)}\Bigl|  \nabla_{\tan, u}f(\xi)\Br|\Br\|_{L^p(G; d\xi)} \leq C(\mu, L, p)  n \|f\|_{L^p(G_\ast)}.
    \end{equation}
\end{thm}

Our goal in this section is to prove Theorem \ref{thm-3-2}. We need the following lemma.

        \begin{lem}\label{lem-3-2}  If  $1\leq j\leq d-1$,  $0<p\leq \infty$ and $\ga=\f 1 \al-\f12$, then for any      $f\in \Pi_n^{d}$,
          we have
        $$ \Bl(\int_G  (\da_n(x,y))^{\ga p} | \p_{\xi_j(x)} f(x,y)|^p\, dxdy\Br)^{1/p} \leq C_p L^{\f 1 \al}  n \|f\|_{L^p(G_\ast)}.$$
    \end{lem}

\begin{proof}  For simplicity, we assume that $d=3$ and $p<\infty$. (The proof below works equally well for $d>3$ or $p=\infty$.) Without loss of generality, we may also assume that $j=1$.
    First, by Fubini's theorem, we have
    \begin{align}\label{2-2-0:bern}
    \iiint\limits_G & (\da_n(x,y))^{\ga p} \Bl|\p_{\xi_1(x_1, x_2)}  f(x_1, x_2, y)\Br|^pdx_1dx_2dy =\int_{0}^1  I(x_2) \, dx_2,
    \end{align}
    where
    $$ I(x_2) :=  \int_{0}^1  \int_{g(x_1, x_2)-1}^{g(x_1, x_2)} \Bl|\bigl(\p_1 +\p_1 g(x_1, x_2) \p_3\bigr) f(x_1, x_2, y)\Br|^p (g(x_1, x_2)-y+n^{-2})^{\ga p}\, dy \, dx_1.$$
    For each fixed $x_2\in [0,1]$, applying Theorem \ref{thm-2-2} to the function $g(\cdot, x_2)$ and the polynomial $f(\cdot, x_2, \cdot)$, we obtain
    \begin{align*}
    I(x_2) & \leq C L^{p/\al} n^{p}  \int_{-1}^{2} \int_{g(x_1, x_2)-8}^{g(x_1, x_2)} |f (x_1,  x_2, y)|^p\, dy dx_1.
    \end{align*}
    Integrating this last inequality over $x_2\in [-1,2]$,  we deduce  the stated estimate in Lemma \ref{lem-3-2}.
\end{proof}

%
%
%
%
%
%

    \begin{proof}[Proof of Theorem \ref{thm-3-2}]

    For $(x,y)\in G$ and $u\in D_1$,  we have
\begin{equation}\label{10-8-eq}  \p_{\xi_j(u)} =\p_{\xi_j(x) } +\Bl( \p_j g(u)-\p_j g(x)\Br) \p_d.\end{equation}
It follows that  for each $(x, y)\in G$,
\begin{align*}
\max_{u\in \Xi_{n, \mu,\al} (x, y) } \Bigl|  \p_{\xi_j (u) } f(x,y)\Br|&\leq  \Bigl|  \p_{\xi_j(x)} f(x,y)\Br|+C L \|u-x\|^{\al-1} |\p_d f(x,y)|\\
&\leq  \Bigl|  \p_{\xi_j(x)} f(x,y)\Br|+C_\mu  L  \da_n(x,y)^{(\al-1)/\al}|\p_d f(x,y)|.
\end{align*}

Using Lemma \ref{lem-3-2}, we have
    $$ \Bl(\int_G  (\da_n(x,y))^{\ga p} | \p_{\xi_j(x)} f(x,y)|^p\, dxdy\Br)^{1/p} \leq C_p L^{\f 1 \al}  n \|f\|_{L^p(G_\ast)}.$$
On the other hand,  we have
\begin{align*}
& \Bl( \int_G |\da_n (x, y)|^{(1-\f 1\al) p +\ga p} |\p_d f(x,y)|^p dxdy \Br)^{1/p}\\
& = \Bl( \int_{D_1} \int_{g(x)-1}^{g(x)} (\sqrt{\da_n(x,y)})^p |\p_df(x,y)|^p\, dy dx\Br)^{1/p}\leq C(p) n \|f\|_{L^p(G_\ast)},
\end{align*}
where we used the Bernstein-Markov  inequality for univariate  algebraic polynomials in the last step.
Putting the above together, and noticing that
$$ |\nabla_{\tan, u} f|\sim \max_{1\leq j\leq d-1} |\p_{\xi_j(u)} f|,$$
we obtain the desired estimate.
    \end{proof}

\section{Tangential $L^p$ Bernstein-Markov inequalities on general $C^\al$ domains}

 In this section, we will
extend
the tangential  Bernstein type inequality  established in the last section to a  more general $C^\al$-domain with $1\leq \al\leq 2$. Our goal is to prove Theorem \ref{thm-1-1A}  (see Theorem \ref{thm-5-2} and Corollary\ref{cor-5-3} below).

Let us first write  the tangential Bernstein inequality in Theorem \ref{thm-3-2}  in an equivalent form, which  can be easily extended to a general $C^\al$-domain.  With  the notations and assumptions of the last section, we have  the following lemma, which will play a crucial role in our discussion.

\begin{lem}\cite[Lemma 3.2]{DaiPry}\label{lem-9-1} Let  $\Ga'=\{ (x, g(x)):\  \   x\in D_2\}$.  If $\xi=(\xi_x, \xi_y)\in  G$, then
    $$c_\ast (g(\xi_x)-\xi_y)\leq \dist(\xi, \Ga')\leq g(\xi_x)-\xi_y,$$
    where $\dist(\xi, \Ga') =\inf_{\eta\in \Ga'} \|\xi-\eta\|$, and
    $ c_\ast =  \f 1 {3 \sqrt{ 1+\|\nabla g\|_\infty^2}}$
    and   $\|\nabla g\|_\infty=\max_{x\in D_2}\|\nabla g(x)\|$.
\end{lem}

Lemma \ref{lem-9-1} was proved in \cite[Lemma 3.2]{DaiPry} for $C^2$-domains, but the proof there uses only the $C^1$ condition. For the sake of  completeness, we copy the proof here.

\begin{proof}
    Let  $\xi=(\xi_x, \xi_y)\in G$. Since $(\xi_x, g(\xi_x))\in \p' G\subset \Ga'$, we have
    $$\dist(\xi, \Ga')\leq \|(\xi_x,\xi_y)-(\xi_x, g(\xi_x))\|=g(\xi_x)-\xi_y.$$

    It remains to prove  the inverse inequality,
    \begin{equation}\label{3-4a}
    \dist(\xi, \Ga')\ge     c_\ast (g(\xi_x)-\xi_y).
    \end{equation}
    Let   $(x, g(x))\in\Ga'$ be such that
    $$\dist(\xi, \Ga') =\|\xi-(x,g(x))\|.$$
    Since
    $$ \dist(\xi, \Ga')\ge \| x-\xi_x\|,$$
    \eqref{3-4a} holds trivially if  $\|x-\xi_x\|\ge   c_\ast (g(\xi_x)-\xi_y)$. Thus, without loss of generality, we may assume that
    $\|x-\xi_x\|< c_\ast (g(\xi_x)-\xi_y)$. We then write
    \begin{align}
    \|\xi-(x,g(x))\|^2=&\|\xi_x- x\|^2 +|\xi_y -g(\xi_x)|^2+\notag\\
    &+  |g(\xi_x)-g(x)|^2
    +2 (\xi_y-g(\xi_x))\cdot (g(\xi_x)-g(x)).\label{eq-9-2}
    \end{align}
    Since  $\|x-\xi_x\|\leq  c_\ast (g(\xi_x)-\xi_y)$, we have
    \begin{align*}
    & \|\xi_x- x\|^2+
    |g(\xi_x)-g( x)|^2+2 (\xi_y-g(\xi_x))\cdot (g(\xi_x)-g( x))\\
    &\leq (1+\|\nabla g\|_\infty^2) \|\xi_x-x\|^2 +2\|\nabla g\|_\infty (g(\xi_x)-\xi_y)\| \xi_x- x\|\\
    &\leq \Bl[c_\ast^2 (1+\|\nabla g\|_\infty^2)+ 2 \|\nabla g\|_\infty c_\ast\Br] (g(\xi_x)-\xi_y)^2\leq \f 79 (g(\xi_x)-\xi_y)^2.
    \end{align*}
    Thus, using~\eqref{eq-9-2},   we obtain
    $$
    \dist(\xi, \Ga')^2= \|\xi-(x,g(x))\|^2
    \ge \f29 |\xi_y -g(\xi_x)|^2,$$
    which implies the desired lower estimate \eqref{3-4a}. This completes the proof of Lemma~\ref{lem-9-1}.
\end{proof}


Next, we recall that for  $u\in D_1$,
\[ |\nabla_{\tan, u} f(\xi)|=\max_{\eta\in S_u\cap \sph} |\p_\eta f(\xi)|,\   \ \xi\in G.\]
For convenience, we also define 
\[|\nabla_{\tan, (u, g(u))} f(\xi)|:= |\nabla_{\tan, u} f(\xi)|,\  \ u\in D_1,\   \xi\in G.\]
Let $\Ga'=\p'G_\ast$.  By Lemma \ref{lem-9-1}, we have  that  for  each $\xi=(\xi_x, \xi_y)\in G$,
\[ \vi_n(\xi): =\sqrt{g(\xi_x)-\xi_y}+\f1n\sim \sqrt{\dist(\xi, \Ga')} +\f 1n =: \vi_{n, \Ga'} (\xi).\]
Recall also that  for  $\mu>1$,
$$\Xi_{n, \mu,\al} (\xi):= \Bl\{ u\in D_1:\  \  \|u-\xi_x\|\leq \mu |\vi_n(\xi)|^{2/\al}\Br\}.$$
If  $u\in \Xi_{n, \mu,\al} (\xi)$  and $\xi=(\xi_x, \xi_y)$,  then  $\eta=(u, g(u))\in\p' G$, and
\begin{align*}
\|\eta-\xi\|&\leq \|u-\xi_x\| + |g(u)-g(\xi_x)| +|g(\xi_x) -\xi_y|\\
&\leq C \Bl(\|u-\xi_x\|+ \dist(\xi, \Ga')\Br)\leq  C \mu \Bl(\vi_{n, \Ga'} (\xi)^{2/\al} + \dist(\xi, \Ga')\Br).
\end{align*}
Since $\Ga'$ is a bounded set and $\al\ge 1$, we have
\[  \dist(\xi, \Ga')\leq C \dist(\xi, \Ga')^{1/\al} \leq C \vi_{n, \Ga'} (\xi)^{2/\al} .\]
Thus,
\begin{align*}
\Bl\{ (u, g(u)):\  \  u\in \Xi_{n, \mu,\al} (\xi)\Br\}\subset  \Bl\{ \eta\in\Ga':\  \  \|\eta-\xi\|\leq C\mu \vi_{n,\Ga'}(\xi)^{2/\al}\Br\}.
\end{align*}
The corresponding inverse relation with a possibly   different value of $\mu$ holds as well. Thus, we can reformulate the tangential Bernstein inequality \eqref{3-4} equivalently as follows:
    \begin{equation*}
\Bl\|(\vi_{n,\Ga'}) ^{2\ga} \mathcal{D}_{n,\mu} f \Br\|_{L^p(G)}\leq C_\mu  n \|f\|_{L^p(G^\ast)},
\end{equation*}
where
\begin{equation*}
\mathcal{D}_{n,\mu} f(\xi):=\max\Bl\{ | \nabla_{\tan, \eta} f(\xi)|:\   \ \eta\in\p \Og',\   \  \|\eta-\xi\|\leq \mu \vi_{n,\Ga'}(\xi)^{2/\al}\Br\}.
\end{equation*}

This  last version of tangential Bernstein inequality can be easily extended to a more general $C^\al$ domain.

 In the sequel,  we assume $\Og\subset \RR^d$ is  a compact $C^\al$-domain with $1\leq \al\leq 2$, whose precise definition  is given as follows. For $r>0$ and  $\xi\in\RR^d$, we define $\ell_{\alpha}$ balls by
 \[ B^\al(\xi,r):=\{\eta\in \RR^d:\  \ \|\eta-\xi\|_\al < r\},\]
 where $\|\eta\|_\al:=(\sum_{j=1}^d |\eta_j|^\al)^{1/\al}$ for $\eta=(\eta_1,\cdots, \eta_d)\in\RR^d$. Hence $B^2(\xi,r)$ stands for the usual Euclidian balls.

\begin{defn}\label{def-C2} Let $1\leq \al \leq 2$.
    A bounded   set   $\Og$ in $ \RR^{d}$ is called a $C^\al$-domain    if there exist a positive constant  $\k_0$, and a finite cover of the boundary $\p \Og$  by  open sets $\{ U_j\}_{j=1}^J$ in $\RR^d$ such that\begin{enumerate}[\rm (i)]
        \item   for each $1\leq j\leq J$, there exists a function $\Phi_j\in C^\al(\RR^d)$ such that
        $$U_j\cap \p\Og=\{ \xi\in U_j:\  \ \Phi_j(\xi)=0\}\   \ \text{and}\   \ \nabla \Phi_j(\xi)\neq 0,\   \   \    \   \forall \xi\in U_j\cap \p\Og; $$
        \item for each $\xi\in \p\Og$ there exist affine transforms $A_1, A_2$ of $\R^d$ with $\det A_i\ge \kappa_0$ and $A_i(e_1)=\xi$, $i=1,2$, such that
        $$A_1(B^\al (0,1))\subset {\Og}\   \ \text{and}\     \   A_2(B^\al (0,1)) \subset {\RR^d\setminus \Og},$$
        where $e_1=(1,0,\dots,0)$ is the first standard basis vector.
    \end{enumerate}
\end{defn}

Condition (ii) of this definition is needed to ensure that for any point of the boundary both the domain and its complement contain an $\ell_{\alpha}$ ball with ``vertex'' at this point. This is a generalization of the ``rolling ball'' property~\cite{DaiPry}*{Def.~1.1(ii)} for $\alpha=2$, see also~\cite{DaiPry}*{Rem.~1.2}.

Let $\Ga=\p\Og$.
Denote by $\mathbf{n}_\eta$  the  outer unit normal vector   to  $\Ga$ at  $\eta\in\Ga$.
For  $ \xi\in\Og$,  $ f\in C^\infty(\Og)$,
we define
$$ |\nabla_{\tan, \eta} f(\xi)|:= \max\Bl\{
\bl|  \p_{\pmb{\tau}} f(\xi)\br|:\   \   \pmb\tau\in\sph,\  \ \pmb\tau\cdot \mathbf n_\eta=0\Br\},\  \ \eta\in\Ga. $$
Furthermore, given  a parameter $\mu\ge \Bl( \diam(\Og)+1\Br)^2$, we define
\begin{equation}\label{eqn:D max def}
\mathcal{D}_{n,\mu} f(\xi):=\max\Bl\{  |\nabla_{\tan, \eta} f(\xi)|:\   \ \eta\in\p \Og,\   \  \|\eta-\xi\|\leq \mu \vi_{n,\Ga}(\xi)^{2/\al}\Br\},\   \ \xi\in\Og,
\end{equation}
where
\begin{equation*}
\vi_{n,\Ga}(\xi):=\sqrt{\dist(\xi, \Ga)} +n^{-1},\   \ n=1,2,\dots, \xi\in\Og.
\end{equation*}

With the above notation, we can then state our main tangential $L^p$ Bernstein type inequality on general $C^\al$-domains as follows:

\begin{thm}\label{thm-5-2}
    Let  $\ga=\f 1 \al-\f12$, and   $0<p<\infty$.
    Then for any  $f\in\Pi_n^d$ and parameter $\mu>1$,    we have
    $$  \Bl\| (\vi_{n,\Ga} )^{2\ga} \mathcal{D}_{n,\mu} f\Br\|_{L^p(\Og)} \leq C(\mu,\Og, p)  n  \|f\|_{L^p(\Og)}.$$
\end{thm}

    Since general  $C^{\alpha}$ domains can be covered  by $C^{\alpha}$ domains of special type,
    Theorem \ref{thm-5-2} can be deduced directly from Theorem \ref{thm-3-2} and Lemma \ref{lem-9-1}. Since the proof is  very close to the proof in Section 6 of \cite{DaiPry}, we skip the details here.

Clearly, $\vi_{n,\Ga}(\xi)\geq n^{-1}$ that is using the lower bound $(\vi_{n,\Ga} )^{2\ga}\geq n^{-\frac{2}{\al}+1}$ in the above theorem yields the next tangential $L^p$ Markov type inequality on general $C^\al$-domains:

\begin{cor}\label{cor-5-3}
                For any  $0<p<\infty, f\in\Pi_n^d$ and parameter $\mu>1$,    we have
    $$  \Bl\|  \mathcal{D}_{n,\mu} f\Br\|_{L^p(\Og)} \leq C(\mu,\Og, p)  n^{\frac{2}{\al}}  \|f\|_{L^p(\Og)}.$$
\end{cor}

\section{Sharpness of the tangential $L^p$ Markov inequality on general $C^\al$ domains}

Corollary 5.3 asserts that $L^p$ norms of tangential derivatives of polynomials of degree $n$ on general $C^\al,  1\leq \alpha\leq 2$ domains are of order $n^{\frac{2}{\al}}$. For $p=\infty$ this upper bound is known to be sharp, see \cite{Kroo}. It is considerably harder to verify $L^p$ lower bounds in case when $0<p<\infty$. This is the question we will settle in this section.

Let us consider $C^{\alpha}, 1<\alpha\leq 2$ domain $D\subset \RR^d$. In what follows affine images of
the $\ell_{\alpha}$ unit ball
$ B^\al(0,1)$ will be called $\ell_{\alpha}$ ellipsoids, with images of the standard basis vectors $e_j:=(\delta_{i,j})_{1\leq i\leq d}, 1\leq j\leq d$ being the vertices of these ellipsoids. Recall that when $D\subset \RR^d$ is a $C^{\alpha}, 1<\alpha\leq 2$ domain then for any $x\in\partial D$ on its boundary there exists an \emph{inscribed} $\ell_{\alpha}$ ellipsoid $E\subset D$ for which $x\in\partial E$. Clearly in order to prove lower bounds for tangential Markov type inequality for $C^{\alpha}, 1<\alpha\leq 2$ domains
we need to ensure that this domain is not in a higher Lip$\alpha$ class. We will accomplish this by assuming that the boundary of the domain contains an \textbf{exactly} $C^{\alpha}$ point $y\in\partial D$ such that for some $\delta>0$ there exist a \emph{superscribed} $\ell_{\alpha}$ ellipsoid $E_1$ with vertex at $y$ so that $D\cap B^2(y,\delta)\subset E_1$. With this definition we have the next general converse to the tangential Markov type inequality given by Corollary 5.3.

\begin{thm} Let $1\leq p<\infty, n\in \NN$.  Assume that  $C^{\alpha}, 1<\alpha\leq 2$ domain $D\subset \RR^d$ contains an exactly $C^{\alpha}$ point. Then there exists
 $f\in\Pi_n^d$ such that
    $$  \|  \mathcal{D}_{n} f\|_{L^p(D)} \geq c(D, p)  n^{\frac{2}{\al}}  \|f\|_{L^p(D)}.$$
\end{thm}

\textbf{Proof.} Performing a proper affine map we can assume without the loss of generality that $e_d=(0,...0,1)\in\partial D$ is an exactly $C^{\alpha}$ point on the boundary and
$D\cap B^2(e_d,\delta)\subset B^\al(0,1)$ is the corresponding superscribed $\ell_{\alpha}$ ellipsoid. In addition,
there exists an inscribed $\ell_{\alpha}$ ellipsoid $E\subset D$ for which $e_d\in\partial E$. The existence of such inscribed ellipsoid $E$ is ensured by condition~(ii) of Definition~\ref{def-C2}. Then evidently $e_d$ must be the outer normal to $E$ at $e_d$ and it easily follows that
for any $(x,y)\in D\cap B(e_d,\delta), \;x=(x_1,...,x_{d-1})\in\RR^{d-1}, y\in \RR$ we have with proper $c_1, c_2>0$
\begin{equation}\label{2a}
c_1(|x_1|^{\alpha}+...+|x_{d-1}|^{\alpha})\leq 1-y\leq c_2(|x_1|^{\alpha}+...+|x_{d-1}|^{\alpha}).
\end{equation}
Consider the  polynomial
$$Q(x,y)=x_1J_n(y)g_n(x,y)\in \Pi^d_{(2b+1)n+1}, \;x=(x_1,...,x_{d-1})\in\RR^{d-1}, y\in \RR$$
 where $J_n(y)=J_n^{(\beta,\beta)}(y), y\in \RR$ is the $n$-th Jacobi polynomial with parameter $\beta$ to be specified below and we set $g_n(x,y):=\left(1-\frac{|x|^2-(1-y)^2}{T^2}\right)^{bn}$ where $T$ stands for the diameter of the domain $D$ and the integer $b\in \NN$ will be specified below. Note that since $e_d=(0,...0,1)\in\partial D$ it follows that $|g_n|\leq 1$ on $D$.

 Then by \cite[4.21.7, p.63]{sz},
 \begin{equation}\label{3}
J_n'(y)= (n/2+\beta+1/2)J_{n-1}^{(\beta+1,\beta+1)}(y).
\end{equation}

 Let us give a lower bound for $\int_D| \mathcal{D}_{n}Q|^p$. Clearly it follows from (\ref{2a}) the boundary $\partial D$ in $B(e_d,\delta)$ is given by the surface $y=f(x), \nabla f\in $ Lip$(\alpha-1)$ with
 \begin{equation}\label{3a}
1-c_2(|x_1|^{\alpha}+...+|x_{d-1}|^{\alpha})\leq f(x)\leq 1-c_1(|x_1|^{\alpha}+...+|x_{d-1}|^{\alpha}), \;\;\nabla f(0)=0.
\end{equation}
Moreover, since $(\nabla f,-1)$ is a normal to $\partial D$ it follows that the tangent plane is spanned by $e_j+\frac{\partial f}{\partial x_j}e_d, 1\leq j\leq d-1$. Therefore
setting $\partial^1g:= \frac{\partial g}{\partial x_1}+\frac{\partial f}{\partial x_1}\frac{\partial g}{\partial x_d}$ we have that $| \mathcal{D}_{n}g|\geq |\partial^1g|, \forall g.$ Furthermore
$$\partial^1Q(x,y)=\partial^1(x_1J_n(y))g_n(x,y)+x_1J_n(y)\partial^1g_n(x,y)=$$ $$\left(J_n(y)+\frac{\partial f}{\partial x_1}J_n'(y)x_1\right)g_n(x,y)+x_1J_n(y)\partial^1g_n(x,y).$$

  Set $D_a:=\{(x,y)\in D: 1-\frac{a}{n^2}\leq y\leq 1\}$ with $0<a<1$ to be properly chosen below. First let us note that in $D_a$ we have that $1-y\sim\frac{c}{n^2}, |x|\sim \frac{c}{n^{2/\alpha}}$ yielding $g_n\sim c, |\partial^1g_n|\leq cn.$ Moreover by (\ref{3a}) we have for any $(x,y)\in D_a, |x|\leq c_3(1-y)^{1/\alpha}\leq c_3\left(\frac{a}{n^2}\right)^{1/\alpha}$ and $\left|\frac{\partial f}{\partial x_1}x_1\right|\leq c|x|^{\alpha}$. Hence by the previous relation
  $$| \mathcal{D}_{n}Q|\geq |\partial^1Q| \geq c|J_n(y)|-c\left|\frac{\partial f}{\partial x_1}J_n'(y)x_1\right|-cn|x_1J_n(y)|\geq c|J_n(y)|-\frac{ca}{n^2}|J_n'(y)|-ca^{1/\alpha}|J_n(y)|.$$

  Now we need to recall that $\|J_n^{(\beta,\beta)}\|_{C[-1,1]}\sim n^{\beta}~\sim J_n^{(\beta,\beta)}(1).$ Therefore by the Markov inequality and relation (\ref{3}) it follows that uniformly with respect to $a\in [0, 1]$ and any $y\in [1-\frac{a}{n^2}, 1]$ we have
$$|J_n(y)|\sim n^{\beta}, \;\;\;|J'_n(y)|\sim n^{\beta+2}.$$
Hence we can properly choose $a>0$ so that the above estimate yields
   $$ |\mathcal{D}_{n}Q|\geq |\partial^1Q| \geq cn^{\beta}, \;\;\;(x,y)\in D_a.$$
   Thus we obtain the next lower bound for the integral of tangential derivative
    \begin{equation}\label{4a}
   \int_D| \mathcal{D}_{n}Q|^p\geq \int_{D_a}| \mathcal{D}_{n}Q|^p\geq cn^{\beta p}n^{-2+\frac{-2d+2}{\alpha}}.
  \end{equation}

  Next we need an upper bound for $\int_D|Q|^p$. Since  $0\leq g_n\leq 1, (x,y)\in B(e_d,\delta)$ it follows by (\ref{2a}) that
 $$\int_{D\cap B(e_d,\delta)}|Q|^p\leq \int_{D\cap B(e_d,\delta)}|x|^p|J_n(y)|^p\leq \int_{1-\delta}^1\int_{|x|\leq c_3(1-y)^{1/\alpha}}|x|^p|J_n(y)|^pdxdy$$
 $$\leq c \int_{1-\delta}^1(1-y)^{\frac{p+d-1}{\alpha}}|J_n(y)|^pdy.$$

  Now  we will apply asymptotic properties of the Jacobi polynomial $J_n=J_n^{(\beta,\beta)}(x), \beta>-1$ verified in \cite{sz}, (7.34.1), (7.34.4), p. 173. It is essentially shown there that
\begin{equation}\label{sz}
\int_{\delta}^1(1-x)^{\mu}|J_n(x)|^pdx\sim n^{-2\mu-2+\beta p}, \;\; 2\mu<\beta p-2+\frac{p}{2}.
\end{equation}
 (In fact, in \cite{sz} these asymptotic relations are verified for $p=1$ but they follow analogously for any $p\geq 1$.) Using this result with $\mu:=\frac{p+d-1}{\alpha}$ and any $\beta >2d+2$ we obtain from the previous estimate
  \begin{equation}\label{5a}
\int_{D\cap B^2(e_d,\delta)}|Q|^p\leq c n^{\frac{-2p-2d+2}{\alpha}-2+\beta p}.
\end{equation}

On the other hand using that $0\leq g_n\leq (1-\delta^2)^{bn}, (x,y)\in D\setminus B^2(e_d,\delta)$ and $|x_1J_n|\leq M^n, (x,y)\in D$ with some $M>0$ depending on the domain $D$ it follows that $|Q(x,y)|=|x_1J_n(y)g_n(x,y)|\leq (1-\delta^2)^{bn}M^n, (x,y)\in D\setminus B^2(e_d,\delta).$ Hence we can choose a proper
$b>0$ so that $|Q(x,y)|\leq \gamma^n, \;(x,y)\in D\setminus B^2(e_d,\delta)$ with some $0<\gamma<1$. Combining this observation with the upper bound (\ref{5a}) obviously implies that
\begin{equation}\label{6a}
\int_{D}|Q|^p\leq c n^{\frac{-2p-2d+2}{\alpha}-2+\beta p}.
\end{equation}

Finally, lower bound (\ref{4a}) together with the upper bound (\ref{6a}) yield that
$$\frac{\int_D| \mathcal{D}_{n}Q|^p}{\int_{D}|Q|^p}\geq cn^{\frac{2p}{\alpha}}.$$
Now taking the $p$-th root of the last estimate completes the proof of the theorem.

    \section{Marcinkiewicz type inequalities  on $C^\al$ domains: the case of $d=2$}

    Given a subspace $U\subset L^p(K)$ the Marcinkiewicz-Zygmund type problem for $1\leq p<\infty$ consists in finding discrete point sets $Y_N=\{x_1,...,x_N\}\subset K$ and corresponding positive weights $w_j>0, 1\leq j\leq N$  such that for any $g\in U$ we have
 \begin{equation}\label{mz}
c_1\sum_{1\leq j\leq N}w_j|g(x_j)|^p\leq \|g\|^p_{L^p(K)}\leq c_2\sum_{1\leq j\leq N}w_j|g(x
_j)|^p
 \end{equation}
with some constants $c_1, c_2>0$ depending only on $p,d$ and $K$. In case when $p=\infty$ the above relation is replaced by
$$\|g\|_{C(K)}\leq c\max_{1\leq j\leq N}|g(x_j)|,  \;\;g\in U.$$
  Evidently we must have $N\ge \dim U$ in order for these estimates to be possible for every $g\in U$.  These equivalence relations turned out to be an effective tool used for the discretization of the $L^p$ norms of trigonometric polynomials which is widely applied in the study of the convergence of Fourier series, Lagrange and Hermite interpolation, positive quadrature formulas, scattered data interpolation, see for instance \cite{l} for a survey on the univariate Marcinkiewicz-Zygmund type inequalities. Naturally it is crucial to find discrete point sets $Y_N=\{x_1,...,x_N\}\subset K$ of possibly minimal cardinality $N$. When $U=\Pi^d_n$, and the domain  $K\subset \RR^d$ has nonempty interior, we have dim $\Pi^d_n\sim n^d$ and therefore asymptotically optimal
discrete points sets for $\Pi^d_n$ must be of order $n^d$. For $p=\infty$ such discrete points sets are called optimal meshes. Some general assertions concerning the existence of proper discrete point sets of cardinality $\sim n^d\log^m n$ can be found in \cite{daietal} for $1\leq p\leq 2$ $(m=3)$\footnote{Note that if $g$ is an algebraic polynomial of degree $\le n/2$, then $g^2$ is an algebraic polynomial of degree $\le n$. Hence, if~\eqref{mz} is valid for algebraic polynomials of degree $\le n$ and some $p$, then~\eqref{mz} is valid for algebraic polynomials of degree $\le n/2$ and $2p$. Therefore, the Marcinkiewicz-Zygmund type inequalities for algebraic polynomials for $1\le p\le 2$ automatically extend to the full range $1\le p<\infty$ with the same order of dependence of $N$ on $n$.}, and \cite{bosetal} for $p=\infty$ ($m=d$).  However, besides involving extra log factors the above mentioned results do not provide \emph{explicit} construction for discrete point sets. On the other hand, an  application of tangential Bernstein type inequalities makes it possible to give \emph{explicit} construction of discrete point sets of \emph{asymptotically optimal cardinality}. This approach was used in \cite{Kroo1} in case when $p=\infty$ where existence of optimal meshes was verified for any $C^\al, 2-\frac{2}{d}<\al\leq 2$ domain. In addition, tangential Bernstein type inequalities were applied in  \cite{DaiPry} in order to verify $L^p, 1\leq p<\infty$ Marcinkiewicz type inequalities for $\Pi^d_n$ with $\sim n^d$ points in $C^2$ domains. The aim of the last two sections of this paper is to apply the tangential $L^p$ Bernstein-Markov inequalities verified on general $C^\al$ domains in the previous sections in order to provide asymptotically optimal Marcinkiewicz type inequalities for  $C^\al, 2-\frac{2}{d}<\al<2$ domains.

In this section we will prove Theorem \ref{thm-1-2A} for  the case $d=2$, which we reformulate as follows.

    \begin{thm} \label{thm-1-1:MZ}If  $\Og\subset \RR^2$ is a compact  $C^{\alpha}$- domain with $1<\alpha\leq 2$, then  for any positive integer $n$, there exists a partition $\Og=\cup_{1\leq j\leq N}\Og_j$  of $\Og$ with $N\leq cn^2$ such  that for every $(x_j, y_j)\in \Og_j$, each  $f\in  \Pi^2_n$ and $1\leq p<\infty$,  we have
    $$\frac{1}{2}\sum_{j=1}^N |\Og_j||f(x_j, y_j)|^p\leq \iint_{\Og}|f(x,y)|^pdxdy\leq 2\sum_{j=1}^N|\Og_j||f(x_j, y_j)|^p.$$
    \end{thm}

\begin{rem}
    For $\al=1$,  the stated  result with $N\leq c n^2 \log n$ remains true, as can be seen from the proof below.
\end{rem}

    Since general  $C^{\alpha}$ domains can be covered  by $C^{\alpha}$ domains of special type, it suffices to consider  domains of special type.

     Let   $g: [-2,2]\to \RR$  be   a continuously differentiable function on $[-2,2]$ satisfying       $$|g'(x+t)-g'(x)|\leq L |t|^{\al-1}\   \   \text{ whenever $x, x+t \in [-2,2]$},$$
     where $L>1$ is a constant.
    Define  \begin{align*}
    G:&=\Bl\{(x,y):\  \ 0\leq x\leq 1,\   \ g(x)-\f 14\leq y \leq g(x)\Br\},\end{align*}and
    \begin{align*}
    G_\ast:&=\Bl\{(x,y):\  \ -2\leq x\leq 2, \   \  g(x)-2\leq y \leq g(x)\Br\}.
        \end{align*}
        For $(x,y)\in G_\ast$, we set $\da(x,y):=g(x)-y$ and
        $$  \da_n(x,y)=\da(x,y)+\f 1{n^2}=g(x)-y +\f 1 {n^2}.$$
    As verified in Theorem 3.1 for any   $f\in\Pi_n^2$ and $0<p<\infty$,
    \begin{equation}\label{1-1-Berns}
    \Bl(\int_0^1 \int_{g(x)-1}^{g(x)}   (\da_n(x,y))^{\ga p} | \p_{\tau_x} f(x,y)|^p\, dydx\Br)^{1/p} \leq C_p L^{\f 1 \al}  n \|f\|_{L^p(G_\ast)},
    \end{equation}
    where
    $\ga:=\f 1 \al-\f12$, and   $\tau_x:=(1, g'(x))$.

    To formulate our main proposition, we define, for a bounded function $f$ on $G_\ast$  and  a set $A\subset G_\ast$,
    \[ \osc(f; A) :=\sup_{\xi, \xi'\in A} |f(\xi)-f(\xi')|.\]
    Using  the Bernstein inequality \eqref{1-1-Berns}, we may
    obtain

    \begin{prop}\label{prop-1-2}  For  any $0<\epsilon<1$, $1<\al\leq 2$ and  $n\in\NN$,  there exists a partition  $G=\cup_{j=1}^N G_j$  with $N\leq c_\va n^2$ so that for every $f\in  \Pi^2_n$ and $1\leq p<\infty$,
        \begin{equation}\label{7-3-a}
    \sum_{j=1}^N |G_j| | \osc(f; G_j)|^p\leq \epsilon^p \iint_{G_*}|f(x,y)|^pdxdy.
        \end{equation}
    \end{prop}

    Using  the argument of  \cite{DaiPry},
        Theorem \ref{thm-1-1:MZ}  follows directly from Proposition  \ref{prop-1-2}. 
         Indeed, following the proof of Lemma 7.5 of  \cite{DaiPry}, and using  \eqref{7-3-a} and the fact that $\Og$ can be covered by finitely many domains of special type, we can find a partition  $\Og=\cup_{j=1}^N \Og_j$  with $N\leq c_\va n^2$ so that for every $f\in  \Pi^2_n$ and $1\leq p<\infty$,
         \begin{equation}\label{7-3-a}
         	\sum_{j=1}^N |\Og_j| | \osc(f; \Og_j)|^p\leq  C(\og) \epsilon^p \iint_{\Og}|f(x,y)|^pdxdy.
         \end{equation}
   We then use  the following elementary inequality,
         \[ |a^p-b^p|\leq C_p \da^{1-p} |a-b|^p +\da b^p,\    \ \forall a, b>0,\   \ \da\in (0, 1),\]
         and obtain that for any $(x_j, y_j)\in G_j$ and any $f\in \Pi_n^2$,  
         \begin{align*}
        & \Bl|	\sum_{j=1}^N  |f(x_j, y_j)|^p |\Og_j|  - \iint_{\Og} |f(x,y)|^p\, dxdy \Br| \\
         &\leq C_p \eps^{1-p} \sum_{j=1}^N  \iint_{\Og_j}   |f(x_j, y_j) -f(x,y)|^pdxdy+ \eps  \iint_\Og |f(x,y)|^p\, dxdy\\
         &\leq C_p(\Og) \va  \iint_\Og |f(x,y)|^p\, dxdy.
         \end{align*} 
      Theorem \ref{thm-1-1:MZ} then follows by choosing $\va\in (0, 1)$ so that    $ C_p(\Og) \va\leq \f12$.

        The rest of this note is devoted to the proof of  Proposition \ref{prop-1-2}.
            We   need the following lemma  for algebraic polynomials of one variable.

    \begin{lem}\label{lem-3-1:bern}
    Let $\be \ge -\f12$ and let  $$x_j:=\f {j^2} {4m^2},\   \   \ \text{  $j=0,1,\cdots, m$. }$$
    Then for every  $f\in\Pi_n^1$ with $n\leq m$,  and each  $1\leq p<\infty$, we have
    \begin{align*}
    &   \Bl(\f 1m \sum_{ j=1}^{m}\bigl( x_j^{\beta+\f12}   +m^{-1} \bigr) \max_{x\in [x_{j-1},x_j]} |f(x )|^p\Br)^{\f1p}\leq C \Bl(\int_{0}^{1} |f(x)|^p x^{\beta}\, dx\Br)^{\f1p},
    \end{align*}
    where $C>0$ is a constant depending only on $\beta$.
\end{lem}
\begin{proof} This lemma   can be obtained by a slight modification of the proof of Theorem~3.1 of \cite{MT2}.
Write $$x_j=\sin^2\f {\ta_j}2,\   \   \text{where}\   \  \ta_j\in [0, \pi/3],\   \  j=0,1,\cdots, m.$$  For $f\in \Pi_n^1$, let
    $$T_n(\ta) =f(\sin^2\f \ta 2) =f\Bl( \f {1-\cos \ta}2\Br),\   \ \ta\in\RR.$$
    Then $T_n$ is an even  trigonometric polynomial of degree at most $n$, and $$\f 1{m}\leq \ta_j-\ta_{j-1}\leq  \f {2}m,\   \   \  \text{ $j=1,2,\cdots, m$.}$$

Now setting  $I_j=[\ta_{j-1}, \ta_j]$, we have that   for each $\ta\in I_j$,
    $$ |T_n(\ta)| \leq \int_{I_j} |T_n'(s)|\, ds +\f 1 {|I_j|} \int_{I_j} |T_n(s)|\, ds,$$
    which implies
    \begin{align*}
    \max_{\ta \in I_j} |T_n(\ta)|^p & \leq  C^p m^{-(p-1)} \int_{I_j} |T_n'(s)|^p\, ds+ C^p m \int_{I_j} |T_n(s)|^p\, ds.
    \end{align*}
    Since
    \[  \sin \f {\ta_j}2 \leq \min_{\ta \in I_j} \sin \f \ta 2  +  m^{-1},\  \ 1\leq j\leq m, \]
    it follows that
    \begin{align*}
    &   \Bl(\f 1m \sum_{ j=1}^{m}\bigl( x_j^{\beta+\f12}   +m^{-1} \bigr) \max_{x\in [x_{j-1},x_j]} |f(x )|^p\Br)^{\f1p}=
    \left(\f 1m \sum_{ j=1}^{m}\Bl[ \Bl(  \sin \f {\ta_j}2 \Br)^{2\be+1}  +\f1{m} \Br] \max_{ \ta\in I_j} |T_n(\ta)|^p\right)^{1/p}\\
    &\leq C_\be (S_1+S_2),
    \end{align*}
    where
    \begin{align*}
    S_1:&= m^{-1} \left(\int_{0}^{\pi/3}|T_n'(\ta)|^p \Bl[\Bl(\sin \f \ta 2\Br)^{2\be+1} +\f 1m  \Br]\, d\ta\right)^{1/p},\\
    S_2:&=  \left(\int_0^{\pi/3}  |T_n(\ta)|^p \Bl[\Bl( \sin \f \ta 2\Br)^{2\be+1} +\f 1m \Br]\, d\ta\right)^{1/p}.
    \end{align*}
    For the integral  $S_1$,  we use   the  weighted Bernstein inequality for trigonometric polynomials (see \cite[Theorem 4.1]{MT2}) to obtain
    \begin{align*}
    S_1     &\leq  m^{-1} \left(\f 2 {\sqrt{3}}\int_{0}^{\pi/3}|T_n'(\ta)|^p \Bl[\Bl(\sin \f \ta 2\Br)^{2\be+1}\cos \f \ta 2 +\f 1m  \Br]\, d\ta\right)^{1/p}\\
    &\leq   C \left(\int_{0}^\pi |T_n(\ta)|^p \Bl[\Bl(\sin \f \ta 2\Br)^{2\be+1}\cos \f \ta 2 +\f 1m  \Br]\, d\ta\right)^{1/p},
    \end{align*}
which, using the  Schur-type inequality for trigonometric
polynomials (see \cite[(3.3)]{MT2}),  is estimated above by
    \begin{align*}
     C \Bl(\int_0^\pi  |T_n(\ta)|^p \Bl(\sin \f \ta 2\Br)^{2\be+1} \cos\f \ta 2 \, d\ta\Br)^{1/p}=C\Bl( \int_{0}^1 |f(x)|^p  x^\be  dx\Br)^{1/p}.
    \end{align*}

The integral  $S_2$ can be estimated similarly:
    \begin{align*}
    S_2\leq  \Bl(\f 2 {\sqrt{3}}\int_0^{\pi/3}  |T_n(\ta)|^p \Bl[\Bl(\sin \f \ta 2\Br)^{2\be+1}\cos \f \ta 2 +\f 1m  \Br]  \, d\ta\Br)^{1/p}\leq C\Bl( \int_{0}^1 |f(x)|^p  x^\be  dx\Br)^{1/p}.
    \end{align*}

    Putting the above together, we deduce the inequality stated in the lemma.
\end{proof}

We are now in a position to prove  Proposition \ref{prop-1-2}.

\begin{proof}[Proof of Proposition \ref{prop-1-2}] Let $m>2n$ be an integer such that $\f 1m \sim \f \va n$.  Let   $z_j:= \f {j^2} {4m^2}$ for  $ j=0,1,\dots, m$, and let   $
    x_{i,j}:=
    \f {i}{N_j}$ for $1\leq j\leq m$ and $0\leq i \leq N_j,$
     where $N_j$ is  an integer $\ge n$  to be specified later. We then define a partition $G=\bigcup_{j=1}^m \bigcup_{i=1}^{N_j} I_{i,j}$  as follows:
    \begin{align*}
    I_{i,j}:&=\Bl\{ (x,y)\in G:\  \   x_{i-1,j}\leq  x\leq x_{i,j},\   \  z_{j-1}\leq  g(x)-y \leq z_{j}\Br\},\\
    &\    \    \    \   \  \hspace{5mm}  j=1,2,\cdots, m,\  \ i=1,2,\cdots, N_j.
    \end{align*}
        Note that
    \begin{equation}\label{3-6:bern}
    |I_{i,j}|=(x_{i,j}-x_{i-1,j}) (z_j-z_{j-1})= \f {2j-1}{4m^2N_j},\  \  1\leq i\leq N_j,\   \   1\leq j\leq m.
    \end{equation}
    Our aim is to show that
        \begin{equation}\label{1-4}
    \sum_{j=1}^m \sum_{i=1}^{N_j}  |I_{i,j}| | \osc(f; I_{i,j})|^p\leq \epsilon^p \iint_{G_*}|f(x,y)|^pdxdy,\   \   \ f\in \Pi_n^2.
    \end{equation}

    To show \eqref{1-4}, we define  $F(x,z):= f(x, g(x)-z)$ for $-2\leq x\leq 2$ and $z\in\RR$.
    Then $f(x,y)=F(x, g(x)-y)$, and
    \begin{align*}
    \osc (f; I_{i,j}) &:=\max_{ \sub{ x,x'\in [x_{i-1,j}  , x_{i,j}]\\
            z,z' \in [z_{j-1}, z_{j}]}} |F(x,z)-F(x',z')|\\
    &\leq 2 \sup_{z \in [z_{j-1}, z_{j}]}  \sup_{x\in [x_{i-1,j}, x_{i,j}]} \Bl|F(x, z) -N_j\int_{x_{i-1,j}}^{x_{i,j}}F(u, z_j)\, du\Br|\\
    &   \leq 2 \Bl[ a_{i,j} (f) +b_{i,j}(f)\Br],
    \end{align*}
    where
    \begin{align*}
    a_{i,j}(f):&=     \sup_{\sub{x\in [x_{i-1,j}, x_{i,j}]\\
            z \in [z_{j-1}, z_{j}]}} \Bl|F(x, z) -N_j \int_{x_{i-1,j}}^{x_{i,j} }F(u, z)\, du\Br|,\\
    b_{i,j}(f)&:=N_j \sup_{z\in [z_{j-1}, z_{j}]}\Bl| \int_{x_{i-1,j}}^{x_{i,j} }[F(u, z)-F(u, z_j)]\, du\Br|.
    \end{align*}
    It follows that
    \begin{align*}
        \sum_{j=1}^m \sum_{i=1}^{N_j}  |I_{i,j} | | \osc(f;  I_{i,j})|^p\leq 4^p \Bl( \Sigma_1+\Sigma_2\Br),
    \end{align*}
where
    \begin{align*}
    \Sigma_1&:= \sum_{j=1}^{m}  \sum_{i=1}^{N_j} |I_{i,j}| \bl| a_{i,j}(f)\br|^p\  \ \text{and}\  \
    \Sigma_2:=\sum_{j=1}^{m} \sum_{i=1}^{N_j}  |I_{i,j}|   |b_{i,j} (f)|^p.
    \end{align*}

To estimate the sums $\Sigma_1$ and $\Sigma_2$,  we  first note that
    \begin{align}
    |a_{i,j} (f)|^p&\leq
\Bl(    \int_{x_{i-1,j}}^{x_{i,j}} \sup_{z \in [z_{j-1}, z_{j}]}|\p_1 F(v, z)| dv\Br)^{p}\leq
N_j^{1-p}   \int_{x_{i-1,j}}^{x_{i,j}} \sup_{z \in [z_{j-1}, z_{j}]}|\p_1 F(v, z)|^p dv,\label{1-5} \\
    |b_{i,j}(f)|^p &\leq \Bl(  N_j \int_{z_{j-1}}^{ z_{j}} \int_{x_{i-1,j}}^{x_{i,j}}|\p_2  F(u, z)|\, du dz\Br)^p\notag\\
    &\leq N_j \Bl( \f j  {4m^2}\Br)^{p-1} \int_{z_{j-1}}^{ z_{j}} \int_{x_{i-1,j}}^{x_{i,j}}|\p_2  F(u, z)|^p\, du dz.\label{1-6}
    \end{align}

For the sum $\Sigma_1$, using ~\eqref{3-6:bern} and \eqref{1-5},  we have
    \begin{align}
    \Sigma_1
    &\leq \f {C}{ m^2 } \sum_{j=1}^{m} \sum_{i=1}^{N_j}  \int_{x_{i-1,j}}^{x_{i,j}}  j N_j^{-p}
    \sup_{z \in [z_{j-1}, z_{j}]} |\p_1 F(v, z)|^pdv\notag\\
    &=\f {C}{ m^2 } \int_0^1 \Bl[ \sum_{j=1}^{m}    jN_j^{-p}
    \sup_{z \in [z_{j-1}, z_{j}]} |\p_1 F(v, z)|^p\Br] dv.\label{1-7-0}
    \end{align}
Now we choose $N_j\in\NN$ so  that $$m\leq N_j \sim   m\Bl( \f mj\Br)^{2\ga}\sim \f m {z_j^\ga},$$
where $\ga=\f 1\al-\f12$.
On one hand, since  $1<\al\leq 2$ and $m\sim n$,  we have $2\ga=\f 2 \al -1<1$, and  the number of sets in the partition $G=\bigcup_{i,j} I_{i,j}$ equals to \footnote{This is the only place where the condition  $\al>1$ is required. In the case of $\al=1$,  we have  $2\ga=1$ and $\sum_{j=1}^m N_j\sim n^2\log n$.}
$$\sum_{j=1}^m N_j \leq C m^{2\ga+1} \sum_{j=1}^m j^{-2\ga} \sim m^2\sim n^2.$$
On the other hand,  since   the function $$\p_1 F(v,z)=(\p_1+g'(v)\p_2) f(v, g(v)-z)=(\p_{\tau_v} f )(v, g(v)-z)$$ is an algebraic polynomial of degree at most $n$ in the variable $z$ for each fixed $v\in [-2,2]$, we obtain from Lemma~\ref{lem-3-1:bern} with $\be=\ga p$  that
    for each $v\in [0, 1]$,
    \begin{align*}
&\sum_{j=1}^m   \f 1 {m^2}  j N_j^{-p}
\sup_{z \in [z_{j-1}, z_{j}]} |\p_1 F(v, z)|^p\\
&\leq C^p\f 1{m^p}\sum_{j=1}^m \f {z_j^ {\f 12 +\ga p}}m \sup_{z \in [z_{j-1}, z_{j}]} |\p_1 F(v, z)|^p
\leq C^p \f 1{m^p} \int_0^{1}z^{\ga p}  |\p_1 F(v, z)|^p\, dz\\
&= C^p \f 1{m^p} \int_0^{1}z^{\ga p}  \Bl|(\p_{\tau_v} f )(v, g(v)-z)\Br|^p\, dz =
C^p \f 1{m^p} \int_{g(v)-1}^{g(v)} (g(v)-y)^{\ga p}  \Bl|\p_{\tau_v} f (v, y)\Br|^p\, dy.
    \end{align*}
This together with \eqref{1-7-0} and  the Bernstein inequality \eqref{1-1-Berns}  implies
        \begin{align*}
    \Sigma_1\leq C^p m^{-p}  \int_{0}^1 \int_{g(x)-1}^{g(x)}  |\p_{\tau_x} f(x, y)|^p \da(x,y) ^{\ga p} \, dy\, dx\leq C^p\Bl(\f {n}m\Br)^p  \|f\|_{L^p(G_\ast)}^p. 
    \end{align*}

Finally, for the sum $\Sigma_2$,   using~\eqref{1-6} and~\eqref{3-6:bern}, we obtain
    \begin{align*}
    \Sigma_2
    &\leq  C ^p  \sum_{j=1}^m \sum_{i=1}^{N_j}  \Bl( \f {j} {m^2} \Br)^p   \int_{x_{i-1,j}}^{x_{i,j}}  \int_{z_{j-1}}^{z_{j}}  |\p_2 F(u,z)|^p  dz\, du= \f{C ^p }{m^p}  \sum_{j=1}^m   \Bl( \f {j} {m} \Br)^p   \int_{0}^{1}  \int_{z_{j-1}}^{z_{j}}  |\p_2 F(u,z)|^p  dz\, du \\
        &\leq  \f {C^p} {m^p}  \int_{0}^1\Bl[\sum_{j=1}^m \int_{z_{j-1}}^{z_{j}}  |\p_2 F(u,z)|^p \Bl(\sqrt{z} +\f 1 {m}\Br)^p\, dz\Br]\, du=\f {C^p} {m^p}  \int_{0}^1\Bl[ \int_{0}^{1/4}  |\p_2 F(u,z)|^p \Bl(\sqrt{z} +\f 1 {m}\Br)^p\, dz\Br]\, du\\
    &\leq  C^p m^{-p}  \int_{0}^1 \Bl[ \int_0^{1/4} |\p_2 F(u,z)|^p (\sqrt{z} )^p \, dz\, du+C^pm^{-2p}  \int_{0}^1 \int_0^{1/4} |\p_2 F(u,z)|^p  \, dz\Br]\, du.\end{align*}
    Since $F(u, z)=f(u, g(u)-z)$ is an algebraic polynomial of the variable $z$ of degree at most $n$,
    using   the univariate  Markov-Bernstein-type  inequality  (\cite[Theorem~7.3]{MT2} ),  we obtain
    \begin{align*}
\Sigma_2\leq    &  C^p\Bl(\f nm\Br)^p  \int_{0}^1  \int_0^2 | F(u,z)|^p \, dz\, du \leq C^p \Bl(\f nm\Br)^p \|f\|_{L^p(G_\ast)}^p.\label{3-9:bern}
    \end{align*}

Putting the above together, and taking into account the fact that
$m\sim \f n\va$, we obtain
    \begin{align*}
\sum_{j=1}^m \sum_{i=1}^{N_j}  |I_{i,j} | | \osc(f;  I_{i,j})|^p&\leq 4^p \Bl( \Sigma_1+\Sigma_2\Br)\leq C^p\Bl(\f {n}m\Br)^p  \|f\|_{L^p(G_\ast)}^p\leq \va^p  \|f\|_{L^p(G_\ast)}^p.
\end{align*}
\end{proof}

\section{Marcinkiewicz type inequalities  on $C^\al$ domains: Higher-dimensional case}

In this section, we assume  $d\ge 3$.  The main goal of this last part of the paper is to prove  Theorem \ref{thm-1-2A} for $d\ge 3$,  namely,  the following Marcinkiewicz type inequalities.

\begin{thm} \label{thm-2-1:MZ}If  $\Og\subset \RR^d$ is a compact  $C^{\alpha}$- domain with $2-\frac{2}{d}<\alpha\leq 2$, then  for any positive integer $n$, there exists a partition $\Og=\cup_{1\leq j\leq N}\Og_j$  of $\Og$ with $N\leq c_\al n^d$ such  that for every $\xi_j\in \Og_j$, each  $f\in  \Pi^d_n$ and $d-1< p<\infty$,  we have
    $$\frac{1}{2}\sum_{j=1}^N |\Og_j||f(\xi_j)|^p\leq \int_{\Og}|f(\xi)|^pd\xi\leq 2\sum_{j=1}^N|\Og_j||f(\xi_j)|^p.$$
\end{thm}

Let   $g: [-2d,2d]^{d-1}\to \RR$  be   a continuously differentiable function on $[-2d,2d]^{d-1}$ satisfying       $$|\nabla g(x+t)-\nabla g(x)|\leq L |t|^{\al-1}\   \   \text{ whenever $x, x+t \in [-2d,2d]^{d-1}$},$$
where $L>1$ is a constant.
Define  \begin{align*}
G:&=\Bl\{(x,y):\  \ x\in [0,1]^{d-1},\   \ g(x)-\f 14\leq y \leq g(x)\Br\},\end{align*}and
\begin{align*}
G_\ast:&=\Bl\{(x,y):\  \ x\in [-2d,2d]^{d-1}, \   \  g(x)-2\leq y \leq g(x)\Br\}.
\end{align*}
For $(x,y)\in G_\ast$, we set $\da(x,y):=g(x)-y$ and
$$  \da_n(x,y)=\da(x,y)+\f 1{n^2}=g(x)-y +\f 1 {n^2}.$$

Let $f\in\Pi_n^d$.
Define  $F(x,z):= f(x, g(x)-z)$ for $x\in [-2d, 2d]^{d-1}$ and $z\in\RR$. Then
\[ \p_j F(x,z) = \p_{\xi_j(x)} f(x, g(x)-z),\   \ j=1,2,\cdots, d-1, \]
where $\xi_j(x) = e_j +\p_j g(x) e_d$.
Let  $\ga=\f 1 \al-\f12$ and $1\leq p<\infty$.   Recall that by Theorem 4.1
\begin{equation}\label{6-1}
\int_{[0,1]^{d-1}} \int_0^{\f 14}  z ^{\ga p}
\Bl|\nabla_x F(x, z) \Br|^p\, dz dx  \leq C(\mu, L)^p  n^p \|f\|_{L^p(G_\ast)}^p.
\end{equation}

For the proof of Theorem \ref{thm-2-1:MZ}, it is enough to show

\begin{prop}\label{prop-2-2}  For  any $0<\epsilon<1$,  $\f {2(d-1)}d<\alpha\leq 2$  and  $n\in\NN$,  there exists a partition  $G=\cup_{j=1}^N G_j$  with $N\leq c_{\al,\va}  n^d$ so that for every $f\in  \Pi^d_n$ and $d-1< p<\infty$,
    \begin{equation}\label{l2}
    \sum_{j=1}^N |G_j| | \osc(f; G_j)|^p\leq \epsilon^p \int_{G_*}|f(\xi)|^pd\xi
    \end{equation}
\end{prop}

\begin{proof} Let $m>2n$ be an integer such that $\f 1m \sim \f \va n$.  Let   $z_j:= \f {j^2} {4m^2}$ for  $ j=0,1,\dots, m$. Let $N_j\ge m $ be an integer   such that
    $$ N_j \sim   m\Bl( \f mj\Br)^{2\ga},\  \ j=1,2,\cdots, m. $$  We partition the cube $[0, 1]^{d-1}$ into pairwise disjoint sub-cubes $Q_{i,j}$, $i\in\Ld_j$ of equal  side length $1/N_j$, where  $\# \Ld_j =N_j^{d-1}$.
    We then define a partition $G=\bigcup_{j=1}^m \bigcup_{i\in\Ld_j}  I_{i,j}$  as follows:
    \begin{align*}
    I_{i,j}:&=\Bl\{ (x,y)\in G:\  \  x\in Q_{i,j},\   \  z_{j-1}\leq  g(x)-y \leq z_{j}\Br\},\\
    &\    \    \    \   \  \hspace{5mm}  j=1,2,\cdots, m,\  \ i\in\Ld_j.
    \end{align*}
    Clearly,  the number of sets in the partition $G=\bigcup_{i,j} I_{i,j}$ equals to
    $$\sum_{j=1}^m \#\Ld_j =\sum_{j=1}^m N_j^{d-1} \leq C m^{(2\ga+1)(d-1)} \sum_{j=1}^m j^{-(2\ga)(d-1)} \sim m^d,$$
    where the last step uses the assumption $\f {2(d-1)}d<\alpha$ so that
    $$ 2\ga (d-1) =(\f 2\al -1) (d-1) <1.$$

    Our aim is to show that
    \begin{equation}\label{1-4}
    \sum_{j=1}^m \sum_{i\in\Ld_j}   |I_{i,j}| | \osc(f; I_{i,j})|^p\leq \epsilon^p \int_{G_*}|f(x,y)|^pdxdy,\   \   \ f\in \Pi_n^d,
    \end{equation}
where
    \begin{equation}\label{3-6:bern}
    |I_{i,j}|=|Q_{i,j}| (z_j-z_{j-1})= \f {2j-1}{4m^2N_j^{d-1}},\  \  i\in \Ld_j,\   \   1\leq j\leq m.
    \end{equation}

    Note  that  \begin{align*}
    \osc (f; I_{i,j}) &:=\max_{ \sub{ x,x'\in Q_{i,j}\\
            z,z' \in [z_{j-1}, z_{j}]}} |F(x,z)-F(x',z')|\\
    &\leq 2 \sup_{z \in [z_{j-1}, z_{j}]}  \sup_{x\in Q_{i,j} } \Bl|F(x, z) -\dashint_{Q_{i,j}}F(u, z_j)\, du\Br|\\
    &   \leq 2 \Bl[ a_{i,j} (f) +b_{i,j}(f)\Br],
    \end{align*}
    where
    \begin{align*}
    a_{i,j}(f):&=     \sup_{\sub{x\in Q_{i,j}\\
            z \in [z_{j-1}, z_{j}]}} \Bl|F(x, z) - \dashint_{Q_{i,j}}F(u, z)\, du\Br|,\\
    b_{i,j}(f)&:=\sup_{z\in [z_{j-1}, z_{j}]}\Bl| \dashint_{Q_{i,j} }[F(u, z)-F(u, z_j)]\, du\Br|.
    \end{align*}

    Given a cube $Q$ in $\RR^{d-1}$, we denote by $Q^\ast$ the cube with the same center as $Q$ but $d$ times the length of $Q$.
For the term $a_{i,j}(f)$, using the pointwise   Poinc\'are inequality (see, for instance, \cite[p.11]{Wol}) and H\"older's inequality,  we obtain that
    for  $p>d-1$,
    \begin{align*}
    |   a_{i,j} (f)|^p & \leq \sup_{\sub{x\in Q_{i,j}\\
            z \in [z_{j-1}, z_{j}]}} \Bl(\int_{Q_{i,j}^\ast} \f{|\nabla_u F (u, z)|}{\|u-x\|^{d-2}}\, du \Br)^p \\
    &\leq  C N_j^{-(p-d+1)} \int_{Q_{i,j}^\ast} \sup_{
        z \in [z_{j-1}, z_{j}]} |\nabla_u F (u, z)|^p \, du.
    \end{align*}

    For the term $b_{i,j}(f)$, we have
    \begin{align}
    |b_{i,j}(f)|^p &\leq \Bl(   \int_{z_{j-1}}^{ z_{j}} \dashint_{Q_{i,j}}|\p_d F(u, z)|\, du dz\Br)^p\notag\\
    &\leq N_j^{(d-1)}  \Bl( \f j  {4m^2}\Br)^{p-1} \int_{z_{j-1}}^{ z_{j}} \int_{Q_{i,j}}|\p_d  F(u, z)|^p\, du dz.
    \end{align}
Thus,
    \begin{align*}
    \sum_{j=1}^m \sum_{i\in\Ld_j}   |I_{i,j} | | \osc(f;  I_{i,j})|^p\leq C^p \Bl( \Sigma_1+\Sigma_2\Br),
    \end{align*}
    where
    \begin{align*}
    \Sigma_1&:= \f 1 {m^p} \sum_{j=1}^{m}  \sum_{i\in\Ld_j}  \f {z_j^{\ga p+\f12} } {m}  \int_{Q_{i,j}^\ast} \sup_{
        z \in [z_{j-1}, z_{j}]} |\nabla_u F (u, z)|^p \, du\\
    &\leq  C  \int_{[-d, d]^{d-1}} \Bl[\f 1 {m^p} \sum_{j=1}^{m}    \f {z_j^{\ga p+\f12} } {m} \sup_{
        z \in [z_{j-1}, z_{j}]} |\nabla_u F (u, z)|^p \Br]\, du,\\
    \Sigma_2:&=\sum_{j=1}^{m} \sum_{i\in\Ld_j}   \Bl( \f j  {4m^2}\Br)^{p} \int_{z_{j-1}}^{ z_{j}} \int_{Q_{i,j}}|\p_d  F(u, z)|^p\, du dz=\sum_{j=1}^{m}    \Bl( \f j  {4m^2}\Br)^{p} \int_{z_{j-1}}^{ z_{j}} \int_{[0,1]^{d-1}}|\p_d  F(u, z)|^p\, du dz
    \end{align*}

    Note that for each $1\leq j\leq d-1$,   the function $$\p_j F(v,z)=(\p_j+\p_j g(v)\p_d) f(v, g(v)-z)$$ is an algebraic polynomial of degree at most $n$ in the variable $z$ for each fixed $v\in [-d,d]^{d-1}$, we obtain from Lemma~\ref{lem-3-1:bern} with $\be=\ga p$  that
    for each $u\in [-d, d]^{d-1}$,
    \begin{align*}
    &\f 1{m^p}\sum_{j=1}^m \f {z_j^ {\f 12 +\ga p}}m \sup_{z \in [z_{j-1}, z_{j}]} |\nabla_u  F(u, z)|^p\leq C^p \f 1{m^p} \int_0^{1}z^{\ga p}  |\nabla_u F(u, z)|^p\, dz,\end{align*}
    which, using the tangential Bernstein inequality, implies
\begin{align*}\Sigma_1 & \leq C^p m^{-p}  \int_{[-d,d]^{d-1}}  \int_{g(x)-1}^{g(x)}  |\p_{\tau_x} f(x, y)|^p \da(x,y) ^{\ga p} \, dy\, dx\leq C^p\Bl(\f {n}m\Br)^p  \|f\|_{L^p(G_\ast)}^p. 
    \end{align*}

    Finally, for the sum $\Sigma_2$, we have
    \begin{align*}
    \Sigma_2
    &\leq  \f{C ^p }{m^p}  \sum_{j=1}^m   \Bl( \f {j} {m} \Br)^p   \int_{[0, 1]^{d-1}} \int_{z_{j-1}}^{z_{j}}  |\p_d F(u,z)|^p  dz\, du \\
    &\leq  \f {C^p} {m^p}  \int_{[0,1]^{d-1}} \Bl[\sum_{j=1}^m \int_{z_{j-1}}^{z_{j}}  |\p_d F(u,z)|^p \Bl(\sqrt{z} +\f 1 {m}\Br)^p\, dz\Br]\, du\\
    &=\f {C^p} {m^p}  \int_{[0,1]^{d-1}} \Bl[ \int_{0}^{1/4}  |\p_d F(u,z)|^p \Bl(\sqrt{z} +\f 1 {m}\Br)^p\, dz\Br]\, du\\
    &\leq  C^p m^{-p}  \int_{[0,1]^{d-1}} \Bl[ \int_0^{1/4} |\p_d F(u,z)|^p (\sqrt{z} )^p \, dz\, du+C^pm^{-2p}  \int_{[0,1]^{d-1}} \int_0^{1/4} |\p_d F(u,z)|^p  \, dz\Br]\, du.\end{align*}
    Since $F(u, z)=f(u, g(u)-z)$ is an algebraic polynomial of the variable $z$ of degree at most $n$,
    using   the univariate  Markov-Bernstein-type  inequality  (\cite[Theorem~7.3]{MT2} ),  we obtain
    \begin{align*}
    \Sigma_2\leq    &  C^p\Bl(\f nm\Br)^p  \int_{[0,1]^{d-1}}  \int_0^2 | F(u,z)|^p \, dz\, du \leq C^p \Bl(\f nm\Br)^p \|f\|_{L^p(G_\ast)}^p.\label{3-9:bern}
    \end{align*}

    Putting the above toughener, and taking into account the fact that $m\sim \f n\va$, we obtain
    \begin{align*}
    \sum_{j=1}^m \sum_{i\in\Ld_j}  |I_{i,j} | | \osc(f;  I_{i,j})|^p&\leq  C^p\Bl(\f {n}m\Br)^p  \|f\|_{L^p(G_\ast)}^p\leq \va^p  \|f\|_{L^p(G_\ast)}^p.
    \end{align*}
\end{proof}

\begin{rem}
	For arbitrary $\eps>0$, the Marcinkiewicz-Zygmund inequalities established in \cref{thm-1-1:MZ} and \cref{thm-2-1:MZ} are valid with $\frac12$ and $2$ replaced with $(1-\eps)$ and $(1+\eps)$, respectively, and $c_\alpha$ in the bound on $N$ allowed to depend on $\eps$ as well. The technique of the proof is exactly the same.
\end{rem}

We would like to conclude this paper with the following open question: is it possible to extend \cref{thm-1-1:MZ} to $\alpha=1$ and \cref{thm-2-1:MZ} to $1\le \alpha\le 2-\frac2d$ without allowing any additional logarithmic factors in the bound on $N$?


\begin{bibsection}
    \begin{biblist}
    	
 \bib{Ac}{book}{
 	author={Achieser, N. I.},
 	title={Theory of approximation},
 	note={Translated from the Russian and with a preface by Charles J. Hyman;
 		Reprint of the 1956 English translation},
 	publisher={Dover Publications, Inc., New York},
 	date={1992},
 	pages={x+307},
 	isbn={0-486-67129-1},
 }   	

\bib{bosetal}{article}{
    author={Bloom, Thomas}
    author={Bos, Len},
    author={Calvi, J.~P.},
    author={Levenberg, Norm}
    title={Polynomial interpolation and approximation in $\CC^d$},
    journal={Ann. Polon. Math.},
    volume={106},
    date={2012},
    pages={53-81},
}
\bib{daietal}{article}{
    author={Dai, Feng},
    author={Prymak, Andriy},
    author={Shadrin, Alexei},
    author={Temlyakov, Vladimir},
    author={Tikhonov, Sergei}
    title={Entropy numbers and Marcinkiewicz-type
discretization},
    journal={J. Funct. Anal.},
    volume={281},
    date={2021},
    pages={109090},
}




\bib{DaiPry}{article}{
    author={Dai, Feng},
    author={Prymak, Andriy},
    title={$L^p$-Bernstein inequalities on $C^2$-domains and applications to
        discretization},
    journal={Trans. Amer. Math. Soc.},
    volume={375},
    date={2022},
    number={3},
    pages={1933--1976},
}


\bib{Dit}{book}{
    author={Ditzian, Z.},
    author={Totik, V.},
    title={Moduli of smoothness},
    series={Springer Series in Computational Mathematics},
    volume={9},
    publisher={Springer-Verlag, New York},
    date={1987},
    pages={x+227},
}


\bib{Er}{article}{
    author={Erd\'{e}lyi, Tam\'{a}s},
    title={Notes on inequalities with doubling weights},
    journal={J. Approx. Theory},
    volume={100},
    date={1999},
    number={1},
    pages={60--72},
}

\bib{ke}{article}{
    author={O. D. Kellogg,}
    title={On bounded polynomials in several variables},
    journal={Math. Z.},
    volume={27},
    date={1928},

    pages={55--64},
}

\bib{Kroo}{article}{
    author={Kro\'{o}, Andr\'{a}s},
    title={Markov-type inequalities for surface gradients of multivariate
        polynomials},
    journal={J. Approx. Theory},
    volume={118},
    date={2002},

    pages={235--245},
}

\bib{Kroo1}{article}{
    author={Kro\'{o}, Andr\'{a}s},
    title={Bernstein type inequalities on star--like domains in ${\RR}^d$ with application to norming sets},
    journal={Bull. Math. Sci.},
    volume={3},
    date={2013},
    pages={349-361},
}

\bib{l}{article}{
    author={Lubinsky, Doron},
    title={Marcinkiewicz-Zygmund Inequalities: Methods and Results},
    journal={in Recent Progress in Inequalities (ed. G.V. Milovanovic et al.), Kluwer Academic Publishers, Dordrecht},

    date={1998},

    pages={213--240},
}

\bib{MT2}{article}{
    author={Mastroianni, Giuseppe},
    author={Totik, Vilmos},
    title={Weighted polynomial inequalities with doubling and $A_\infty$
        weights},
    journal={Constr. Approx.},
    volume={16},
    date={2000},

    pages={37--71},
}

\bib{s}{article}{
    author={Sarantopoulos, Yannis},
    title={Bounds on the derivatives of polynomials on Banach Spaces},
    journal={ Math. Proc. Cambr. Philos. Soc.},
    volume={110},
    date={1991},

    pages={307--312},
}

\bib{sz}{book}{
    author={Szeg\H{o}, G\'{a}bor},
    title={Orthogonal polynomials},
    series={American Mathematical Society Colloquium Publications, Vol.
        XXIII},
    edition={4},
    publisher={American Mathematical Society, Providence, R.I.},
    date={1975},
    pages={xiii+432},
}




\bib{Wol}{book}{
    author={Wolff, Thomas H.},
    title={Lectures on harmonic analysis},
    series={University Lecture Series},
    volume={29},
    note={With a foreword by Charles Fefferman and a preface by Izabella
        \L aba;
        Edited by \L aba and Carol Shubin},
    publisher={American Mathematical Society, Providence, RI},
    date={2003},
    pages={x+137},}

    \end{biblist}
\end{bibsection}
\end{document}